\documentclass[12pt,a4paper,reqno]{amsart} 

\usepackage[T1]{fontenc}       

\usepackage{amsfonts}          

\usepackage{amsmath}

\usepackage{amssymb}

\usepackage{overpic}
\usepackage{euscript} 
\usepackage{eurosym}
\usepackage{comment}
\usepackage{mathrsfs} 
\usepackage{ifthen}
\usepackage{esint} 
\usepackage{color}

\long\def\symbolfootnote[#1]#2{\begingroup%
\def\thefootnote{\fnsymbol{footnote}}\footnote[#1]{#2}\endgroup}

{\qed\vspace{5pt}}

\usepackage{amsthm}

\newtheoremstyle{lause}
{5pt}
{5pt}
{\slshape}
{\parindent}
{\bfseries}
{.}
{.5em}
{}

\theoremstyle{lause}

\newtheoremstyle{maaritelma}
{5pt}
{5pt}
{\rmfamily}
{\parindent}
{\bfseries}
{.}
{.5em}
{}

\theoremstyle{maaritelma}
\newtheoremstyle{lause}
{5pt}
{5pt}
{\slshape}
{\parindent}
{\bfseries}
{.}
{.5em}
{}

\theoremstyle{lause}
\newtheorem{theorem}{Theorem}[section]
\newtheorem{lemma}[theorem]{Lemma}
\newtheorem{proposition}[theorem]{Proposition}
\newtheorem{corollary}[theorem]{Corollary}

\newtheoremstyle{maaritelma}
{5pt}
{5pt}
{\rmfamily}
{\parindent}
{\bfseries}
{.}
{.5em}
{}

\theoremstyle{maaritelma}

\newtheorem{example}[theorem]{Example}
\newtheorem{remark}[theorem]{Remark}

\numberwithin{equation}{section}

\begin{document}

\thispagestyle{empty}

\begin{center}

{\large{\textbf{General theory of balayage on locally compact spaces. Applications to weighted minimum energy problems}}}

\vspace{18pt}

\textbf{Natalia Zorii}

\vspace{18pt}

\emph{In memory of Bent Fuglede (8.10.1925\,--\,7.12.2023)}\vspace{8pt}

\footnotesize{\address{Institute of Mathematics, Academy of Sciences
of Ukraine, Tereshchenkivska~3, 02000, Kyiv, Ukraine\\
natalia.zorii@gmail.com }}

\end{center}

\vspace{12pt}

{\footnotesize{\textbf{Abstract.} Under suitable requirements on a kernel on a locally compact space, we develop a theory of inner (outer) balayage of quite general Radon measures $\omega$ (not necessarily of finite energy) onto quite general sets (not necessarily closed). We prove the existence and the uniqueness of inner (outer) swept measures, analyze their properties, and provide a number of alternative characterizations. In spite of being in agreement with Cartan's theory of Newtonian balayage, the results obtained require essentially new methods and approaches, since in the case in question, useful specific features of Newtonian potentials may fail to hold. The theory thereby established extends considerably that by Fuglede (Anal.\ Math., 2016) and that by the author (Anal.\ Math., 2022), these two dealing with $\omega$ of finite energy. Such a generalization enables us to improve substantially our recent results on the Gauss variational problem (Constr.\ Approx., 2024), by strengthening their formulations and/or by extending the area of their validity. This study covers many interesting kernels in classical and modern potential theory, which also looks promising for other applications.}}
\symbolfootnote[0]{\quad 2010 Mathematics Subject Classification: Primary 31C15.}
\symbolfootnote[0]{\quad Key words: Kernels on a locally compact space; energy, consistency, and maximum principles; inner and outer balayage; minimum energy problems with external fields.
}

\vspace{6pt}

\markboth{\emph{Natalia Zorii}} {\emph{General theory of balayage on locally compact spaces and its applications}}

\section{Introduction and main results}\label{sec-intr}

\subsection{Statement of the problem and general conventions}\label{sec-intr1} This paper deals with the theory of potentials on a locally compact (Hausdorff) space $X$ with respect to a {\it kernel} $\kappa$, $\kappa$ being thought of as a symmetric, lower semi\-con\-tin\-uous (l.s.c.) function $\kappa:X\times X\to[0,\infty]$.
To be exact, we are interested in generalizations of the classical theory of balayage on $\mathbb R^n$, $n\geqslant2$ (see e.g.\ \cite{Br}--\cite{Doob}, \cite{L}), to suitable kernels $\kappa$ on $X$.

For a kernel $\kappa$ satisfying the energy, consistency, and domination principles, the theory of {\it inner} balayage of a positive Radon measure $\sigma$ on $X$ of {\it finite} energy to an {\it arbitrary} set $A\subset X$ was originated in the author's paper \cite{Z-arx1}, and it found a further development in \cite{Z-arx-22,Z-arx}. The main tool exploited there was the discovered identity between the inner balayage and the orthogonal projection of $\sigma$ in the pre-Hil\-bert space $\mathcal E$ of all (signed) Radon measures of finite energy onto the strong closure of the class $\mathcal E^+(A)$ of all positive $\mu\in\mathcal E$ concentrated on the set $A$ (see \cite[Theorem~4.3]{Z-arx1}).

It is worth emphasizing here that such an approach, substantially based on the concept of energy and a pre-Hilbert structure on suitable spaces of signed Radon measures, is not covered by the ones developed in the setting of balayage spaces \cite{BH} or $H$-cones \cite{BBC}, both dealing with the theory of {\it outer} balayage.

The main aim of this paper is to extend the theory, established in \cite{Z-arx1}--\cite{Z-arx},
to positive Radon measures $\omega$ on $X$ of {\it infinite} energy, thereby giving an answer to \cite[Problem~7.1]{Z-arx}. Along with an essentially new approach than the one utilized in \cite{Z-arx1}--\cite{Z-arx}, this also requires some additional assumptions imposed on $X$, $\kappa$, $A$, and $\omega$.

In what follows, a locally compact space $X$ is assumed to be {\it second-countable}. Then it is
{\it $\sigma$-compact} (that is, representable as a countable union of compact sets \cite[Section~I.9, Definition~5]{B1}), see \cite[Section~IX.2, Corollary to Proposition~16]{B3}; and hence
negligibility is the same as local negligibility \cite[Section~IV.5, Corollary~3 to Proposition~5]{B2}. (The reader is expected to be familiar with principal concepts of the theory of measures and integration on a locally compact space. For its exposition we refer to Bourbaki \cite{B2} or Edwards \cite{E2}; see also Fuglede \cite{F1} for a brief survey.)

We denote by $\mathfrak M$ the linear space of all (real-valued Radon) measures $\mu$ on $X$, equipped with the {\it vague} (={\it weak\,$^*$}) topology of pointwise convergence on the class $C_0(X)$ of all continuous functions $\varphi:X\to\mathbb R$ of compact support, and by $\mathfrak M^+$ the cone of all positive $\mu\in\mathfrak M$, where $\mu$ is {\it positive} if and only if $\mu(\varphi)\geqslant0$ for all positive $\varphi\in C_0(X)$. The space $X$ being second-countable, every $\mu\in\mathfrak M$ has a {\it countable} base of vague neighborhoods, see \cite[Lemma~4.4]{Z-arx}, and hence any vaguely bounded set $\mathfrak B\subset\mathfrak M$ has a {\it sequence} $(\mu_j)\subset\mathfrak B$ that is vaguely convergent to some $\mu_0\in\mathfrak M$, cf.\ \cite[Section~III.1, Proposition~15]{B2}.

Given $\mu,\nu\in\mathfrak M$, the {\it mutual energy} and the {\it potential} are introduced by
\begin{align*}
  I(\mu,\nu)&:=\int\kappa(x,y)\,d(\mu\otimes\nu)(x,y),\\
  U^\mu(x)&:=\int\kappa(x,y)\,d\mu(y),\quad x\in X,
\end{align*}
respectively, provided the value on the right is well defined as a finite number or $\pm\infty$. For $\mu=\nu$, the mutual energy $I(\mu,\nu)$ defines the {\it energy} $I(\mu,\mu)=:I(\mu)$ of $\mu\in\mathfrak M$.

Throughout this paper, a kernel $\kappa$ is assumed to satisfy the {\it energy principle}, or equivalently to be {\it strictly positive definite}, which means that $I(\mu)\geqslant0$ for all (signed) $\mu\in\mathfrak M$, and moreover that $I(\mu)=0$ only for $\mu=0$. Then all (signed) measures of finite energy form a pre-Hil\-bert space $\mathcal E$ with the inner product $\langle\mu,\nu\rangle:=I(\mu,\nu)$ and the energy norm $\|\mu\|:=\sqrt{I(\mu)}$, cf.\ \cite[Lemma~3.1.2]{F1}. The topology on $\mathcal E$ introduced by means of this norm is said to be {\it strong}.

Another permanent requirement imposed on a kernel $\kappa$ is that it satisfies the {\it consistency} principle, which means that the cone
$\mathcal E^+:=\mathcal E\cap\mathfrak M^+$ is {\it complete} in the induced strong topology, and that the strong topology on $\mathcal E^+$ is {\it finer} than the vague topology on $\mathcal E^+$; such a kernel is said to be {\it perfect} (Fuglede \cite[Section~3.3]{F1}). Thus any strong Cauchy sequence (net) $(\mu_j)\subset\mathcal E^+$ converges {\it both strongly and vaguely} to one and the same unique measure $\mu_0\in\mathcal E^+$, the strong topology on $\mathcal E$ as well as the vague topology on $\mathfrak M$ being Hausdorff.

Besides, we always postulate the {\it domination} and {\it Ugaheri maximum principles}, where the former means that for any $\mu\in\mathcal E^+$ and $\nu\in\mathfrak M^+$ with $U^\mu\leqslant U^\nu$ $\mu$-a.e., the same inequality holds everywhere on $X$; whereas the latter means that there exists $M:=M_{X,\kappa}\in[1,\infty)$, depending on $X$ and $\kappa$ only, such that for each $\mu\in\mathcal E^+$ with $U^\mu\leqslant c_\mu$ $\mu$-a.e., where $c_\mu\in(0,\infty)$, $U^\mu\leqslant Mc_\mu$ on all of $X$. When $M$ is specified, we speak of {\it $M$-Uga\-heri's maximum principle}, and when $M=1$, $M$-Ugaheri's maximum principle is referred to as {\it Frostman's maximum principle} (Ohtsuka \cite[Section~2.1]{O}).

For the {\it inner} and {\it outer} capacities of a set $A\subset X$, denoted by $c_*(A)$ and $c^*(A)$, respectively, we refer to Fuglede \cite[Section~2.3]{F1}. If $A$ is capacitable (e.g.\ open or compact), we write $c(A):=c_*(A)=c^*(A)$. An assertion $\mathcal A(x)$ involving a variable point $x\in X$ is said to hold {\it nearly everywhere} ({\it n.e.}) on $A$ if the set $N$ of all $x\in A$ where $\mathcal A(x)$ fails, is of inner capacity zero. Replacing here $c_*(N)=0$ by $c^*(N)=0$, we arrive at the concept of {\it quasi-everywhere} ({\it q.e.}) on $A$.

For any $A\subset X$, we denote by $\mathfrak C_A$ the upward directed set of all compact subsets $K$ of $A$, where $K_1\leqslant K_2$ if and only if $K_1\subset K_2$. If a net $(x_K)_{K\in\mathfrak C_A}\subset Y$ converges to $x_0\in Y$, $Y$ being a topological space, then we shall indicate this fact by writing
\begin{equation*}x_K\to x_0\text{ \ in $Y$ as $K\uparrow A$}.\end{equation*}

Given $A\subset X$, let $\mathfrak M^+(A)$ denote the set of all $\mu\in\mathfrak M^+$ {\it concentrated on} $A$, which means that $A^c:=X\setminus A$ is $\mu$-negligible, or equivalently that $A$ is $\mu$-mea\-s\-ur\-ab\-le and $\mu=\mu|_A$, $\mu|_A$ being the trace of $\mu$ to $A$, cf.\ \cite[Section~V.5.7]{B2}. (If $A$ is closed, then $\mathfrak M^+(A)$ consists of all $\mu\in\mathfrak M^+$ with support $S(\mu)\subset A$, cf.\ \cite[Section~III.2.2]{B2}.)

Also define $\mathcal E^+(A):=\mathcal E\cap\mathfrak M^+(A)$. As seen from \cite[Lemma~2.3.1]{F1},
\begin{equation}\label{iff}
 c_*(A)=0\iff\mathcal E^+(A)=\{0\}\iff\mathcal E^+(K)=\{0\}\quad\text{for every $K\in\mathfrak C_A$}.
\end{equation}

To avoid trivialities, we suppose throughout the paper that
\begin{equation*}
 c_*(A)>0.
\end{equation*}
While approximating $A$ by $K\in\mathfrak C_A$, we may therefore only consider $K$ with $c(K)>0$.

Yet another permanent condition imposed on a set $A$ is that the cone $\mathcal E^+(A)$ is {\it strongly closed}.\footnote{As shown in \cite[Theorem~2.13]{Z-Oh}, this occurs if $A$ is quasiclosed or quasicompact. By Fuglede \cite[Definition~2.1]{F71}, a set $A\subset X$ is said to be {\it quasiclosed\/} if
\begin{equation*}\label{def-q}
\inf\,\bigl\{c^*(A\bigtriangleup F):\ F\text{ closed, }F\subset X\bigr\}=0,
\end{equation*}
$\bigtriangleup$ being the symmetric difference. Replacing in this definition "closed" by "compact", we arrive at the concept of a {\it quasicompact\/} set. The question whether there exists a set $A\subset X$ that is not quasiclosed, whereas the cone $\mathcal E^+(A)$ is still strongly closed, seems to be open.} Being, therefore, a strongly closed subcone of the strongly complete cone $\mathcal E^+$, the cone $\mathcal E^+(A)$ is {\it complete} in the induced strong topology.

The above-mentioned requirements on a locally compact space $X$, a kernel $\kappa$, and a set $A$ will usually not be repeated  henceforth. Throughout the paper, $M=M_{X,\kappa}$ denotes the constant appearing in Ugaheri's maximum principle.

\begin{remark}
 All those assumptions on $X$ and $\kappa$ are fulfilled by:\footnote{For all these kernels, $M_{X,\kappa}=1$, that is, Frostman's maximum principle actually holds.}
\begin{itemize}
  \item[$\checkmark$] The $\alpha$-Riesz kernels $|x-y|^{\alpha-n}$ of order $\alpha\in(0,2]$, $\alpha<n$, on $\mathbb R^n$, $n\geqslant2$ (see  \cite[Theorems~1.10, 1.15, 1.18, 1.27, 1.29]{L}).
  \item[$\checkmark$] The associated $\alpha$-Green kernels, where $\alpha\in(0,2]$ and $\alpha<n$, on an arbitrary open subset of $\mathbb R^n$, $n\geqslant2$ (see \cite[Theorems~4.6, 4.9, 4.11]{FZ}).
   \item[$\checkmark$] The ($2$-)Green kernel on a planar Greenian set (see \cite[Theorem~5.1.11]{AG}, \cite[Sections~I.V.10, I.XIII.7]{Doob}, and \cite{E}).
\end{itemize}
\end{remark}

\begin{remark}The logarithmic kernel $-\log\,|x-y|$ on a closed disc in $\mathbb R^2$ of radius ${}<1$ is strictly positive definite and satisfies Frostman's maximum principle \cite[Theorems~1.6, 1.16]{L}, and hence it is perfect \cite[Theorem~3.4.2]{F1}. However, the domination principle then fails in general; it does hold only in a weaker sense where the measures $\mu,\nu$ involved in the ab\-ove-quo\-ted definition meet the additional requirement  $\nu(\mathbb R^2)\leqslant\mu(\mathbb R^2)$ \cite[Theorem~II.3.2]{ST}. Because of this obstacle, the theory of balayage on a locally compact space $X$, developed in \cite{Z-arx1}--\cite{Z-arx} and the present paper, does not cover the case of the logarithmic kernel on $\mathbb R^2$.\end{remark}

\subsection{Inner and outer balayage for $\omega\in\mathfrak M^+$}\label{sec-intr2} In what follows, fix $\omega\in\mathfrak M^+$, $\omega\ne0$, such that either $\omega\in\mathcal E^+$, or (a)--(c) hold true, where:
\begin{itemize}
  \item[(a)] $\omega$ {\it is bounded, that is,} $\omega(X)<\infty$.
  \item[(b)] {\it For every compact $K\subset A$, $U^\omega|_K$ is upper semicontinuous} ({\it hence continuous}).\footnote{When speaking of a continuous function, we generally understand that the values are {\it finite} real numbers.}
  \item[(c)] $U^\omega$ {\it is bounded on $A$, i.e.}\footnote{The estimate $m_A>0$ follows from the fact that any strictly positive definite kernel is strictly positive, i.e.\ $\kappa(x,x)>0$ for all $x\in X$ (see \cite[p.~150]{F1}).}
  \begin{equation}\label{ma}
  m_A:=\sup_{x\in A}\,U^\omega(x)\in(0,\infty).
  \end{equation}
  \end{itemize}

It is sometimes convenient to treat $\omega$ as a charge creating {\it the external field}
\begin{equation*}f:=-U^\omega.\end{equation*}
Let $\mathcal E_f^+(A)$ consist of all $\mu\in\mathcal E^+(A)$ for which $f$ is $\mu$-integrable (see \cite{B2}, Chapter~IV, Sections~3, 4); then for
every $\mu\in\mathcal E^+_f(A)$, the so-called {\it Gauss functional\/}\footnote{For the terminology used here, see e.g.\ \cite{L,O}. In constructive function theory, $I_f(\mu)$ is often
referred to as {\it the $f$-weighted energy}, see e.g.\ \cite{BHS,ST}.}
\begin{equation}\label{Gf}
I_f(\mu):=\|\mu\|^2+2\int f\,d\mu=\|\mu\|^2-2\int U^\omega\,d\mu
\end{equation}
is finite. By the Cauchy--Schwarz (Bunyakovski) inequality,
\begin{equation}\label{EEE}
\mathcal E_f^+(A)=\mathcal E^+(A)
\end{equation}
if $\omega$ is of finite energy, while otherwise $\mathcal E_f^+(A)$ includes all bounded $\mu\in\mathcal E^+(A)$, the latter being clear from (c). We denote
\begin{equation}\label{hatw'}
 \widehat{w}_f(A):=\inf_{\mu\in\mathcal E^+_f(A)}\,I_f(\mu)\in[-\infty,0],
\end{equation}
the estimate $\widehat{w}_f(A)\leqslant0$ being caused by the fact that $0\in\mathcal E^+_f(A)$ while $I_f(0)=0$.

Also define
\begin{equation}\label{gamma}
 \Gamma_{A,\omega}:=\bigl\{\nu\in\mathfrak M^+: \ U^\nu\geqslant U^\omega\text{ \ n.e.\ on $A$}\bigr\}.
\end{equation}

\begin{theorem}\label{th1} There is one and the same measure $\omega^A$, called the inner balayage of $\omega$ onto $A$, that is uniquely characterized by any one of the following {\rm(i)--(vi)}.
\begin{itemize}
\item[{\rm(i)}] $\omega^A$ is the only measure in $\mathcal E^+(A)$ having the property
\begin{equation}\label{1}
 U^{\omega^A}=U^\omega\quad\text{n.e.\ on $A$},
\end{equation}
or equivalently
\begin{equation}\label{eqpr3}
\int U^{\omega^A}\,d\mu=\int U^\omega\,d\mu\quad\text{for all $\mu\in\mathcal E^+(A)$}.
\end{equation}
\item[{\rm(ii)}] $\omega^A$ is the only measure in $\mathcal E^+(A)$ satisfying the symmetry relation
\begin{equation}\label{ii}
\int U^{\omega^A}\,d\sigma=\int U^\omega\,d\sigma^A\quad\text{for all $\sigma\in\mathcal E^+$},
\end{equation}
where $\sigma^A$ denotes the only measure in $\mathcal E^+(A)$ with $U^{\sigma^A}=U^\sigma$ n.e.\ on $A$.\footnote{Such $\sigma^A$ does exist; it can be found as the orthogonal projection of $\sigma\in\mathcal E^+$ in the pre-Hil\-bert space $\mathcal E$ onto the convex, strongly complete cone $\mathcal E^+(A)$ (see \cite[Theorem~4.3]{Z-arx1}). Regarding the concept of orthogonal projection in a pre-Hilbert space, see \cite[Theorem~1.12.3]{E2}.}
\item[{\rm(iii)}] $\omega^A$ is uniquely determined within $\mathcal E^+(A)$ by any one of the limit relations\footnote{Assertion (iii) justifies the term "{\it inner} balayage".}
\begin{align}
  \omega^K&\to\omega^A\quad\text{strongly in $\mathcal E^+$ as $K\uparrow A$},\label{c1}\\
  \omega^K&\to\omega^A\quad\text{vaguely in $\mathfrak M^+$ as $K\uparrow A$},\label{c2}\\
  U^{\omega^K}&\uparrow U^{\omega^A}\quad\text{pointwise on $X$ as $K\uparrow A$},\label{c3}
\end{align}
where $\omega^K$ denotes the only measure in $\mathcal E^+(K)$ with $U^{\omega^K}=U^\omega$ n.e.\ on $K$ {\rm(}such $\omega^K$ does exist\/{\rm)}.
\item[{\rm(iv)}] $\omega^A$ is the unique solution to the problem of minimizing $I_f(\mu)$ over $\mu\in\mathcal E^+_f(A)$. That is, $\omega^A\in\mathcal E^+_f(A)$ and
    \begin{equation}\label{hatw}
     I_f(\omega^A)=\min_{\mu\in\mathcal E^+_f(A)}\,I_f(\mu)=\widehat{w}_f(A)\in(-\infty,0].
    \end{equation}
\item[{\rm(v)}] $\omega^A$ is the only measure in $\Gamma_{A,\omega}\cap\mathcal E^+$ having the property
\begin{equation}\label{minpot}
U^{\omega^A}=\min_{\nu\in\Gamma_{A,\omega}}\,U^\nu\quad\text{on $X$}.
\end{equation}
\item[{\rm(vi)}] $\omega^A$ is the only measure in $\Gamma_{A,\omega}\cap\mathcal E^+$ having the property
\begin{equation}\label{minen}
  I(\omega^A)=\min_{\nu\in\Gamma_{A,\omega}}\,I(\nu).
\end{equation}
\end{itemize}
\end{theorem}

\begin{remark}As seen from (\ref{eqpr3}), (\ref{EEE}) is actually valid even if $\omega\not\in\mathcal E^+$. Hence, Theorem~\ref{th1}(iv) remains in force with $\mathcal E^+_f(A)$ replaced throughout by $\mathcal E^+(A)$.\label{rem-int}\end{remark}

\begin{theorem}\label{th1'}If $A$ is Borel, then Theorem~{\rm\ref{th1}} remains valid with "n.e.\ on $A$" replaced throughout by "q.e.\ on $A$". The measure $\omega^{*A}$, thereby uniquely determined, is said to be the outer balayage of $\omega$ onto $A$. Actually, $\omega^{*A}=\omega^A$.\end{theorem}

\begin{remark} The concept of inner (outer) balayage, introduced in Theorems~\ref{th1} and \ref{th1'}, respectively, is essentially in agreement with that of inner (outer) Newtonian balayage by Cartan \cite{Ca2}. However, even for the Newtonian kernel, Theorems~\ref{th1} and \ref{th1'} would fail to hold if the imposed requirements on $\omega$ were omitted. For instance, if $A\subset\mathbb R^n$, $n\geqslant3$, is closed and $\omega:=\varepsilon_y$, where a point $y$ is $2$-regular for $A$ while $\varepsilon_y$ denotes the unit Dirac measure at $y$, then there is {\it no} $\omega^A\in\mathcal E^+(A)$ satisfying (\ref{1}). (Concerning the concept of $2$-regular points for a set $A$, see e.g.\ \cite[Section~V.1.2]{L}.)
\end{remark}

\begin{remark}
Theorems~\ref{th1} and \ref{th1'} would fail to hold if the strong closedness of $\mathcal E^+(A)$ were omitted from the hypotheses. Indeed, in $X:=\mathbb R^n$, $n\geqslant3$, let $\kappa(x,y):=|x-y|^{2-n}$, where $x,y\in\mathbb R^n$, $A:=\{|x|<1\}$, and  $\omega:=\varepsilon_z$, where $|z|>1$. Then, the (Newtonian) balayage $\omega^A$ of $\omega$ to $A$ is a positive measure of finite (Newtonian) energy with $S(\omega^A)=\partial_{\mathbb R^n}A$,\footnote{This can easily be seen from the relationship between Newtonian swept and equilibrium measures, provided by the well-known formula with Kelvin transform involved, see \cite[Section~IV.5.20]{L}.} and hence $\omega^A\not\in\mathcal E^+(A)$.
\end{remark}

\begin{remark}If $A$ is quasiclosed and $\omega\in\mathcal E^+$, the existence of the outer balayage $\omega^{*A}$, characterized by (i) or (iv) with "q.e.\ on $A$" in place of "n.e.\ on $A$", was proved by Fuglede (see \cite[Theorem~4.12]{Fu5}, cf.\ also \cite[Section~6.7]{Fu4}). The methods exploited in the present paper are substantially different from those in \cite{Fu4,Fu5}, which enabled us to generalize Fuglede's results to $\omega$ of infinite energy as well as to establish new characteristic properties of the outer balayage $\omega^{*A}$, given by (ii), (iii), (v), and (vi).\end{remark}

\section{Proofs of Theorems~\ref{th1} and \ref{th1'}}\label{sec-pr-th1}
We begin with some known (or easily verified) facts, often useful in the sequel.

We denote by $\mu_*(\cdot)$ and $\mu^*(\cdot)$ the {\it inner} and {\it outer} measure of a set, respectively, and by $\mathcal U$ the class of all {\it universally measurable} sets $U\subset X$ (that is, measurable with respect to every measure on $X$).

\begin{lemma}\label{l1}
For any $E\subset X$, any $\mu\in\mathcal E^+(E)$, and any $U\in\mathcal U$ with the property $c_*(E\cap U)=0$, we have $\mu^*(E\cap U)=0$.
\end{lemma}

\begin{proof}
Being the intersection of $\mu$-measurable $E$ and universally measurable $U$, the set $E\cap U$ is $\mu$-measurable. Since the space $X$ is $\sigma$-compact, it is therefore enough to show that $\mu_*(E\cap U)=0$, which is however obvious from (\ref{iff}).
\end{proof}

It follows, in particular, that if $U^{\nu_1}=U^{\nu_2}$ n.e.\ on $E\subset X$, where $\nu_1,\nu_2\in\mathfrak M^+$, then for every $\mu\in\mathcal E^+(E)$, the same equality holds true $\mu$-a.e.,\footnote{Here we have utilized the fact that the set $\{x\in X: U^{\nu_1}(x)\ne U^{\nu_2}(x)\}$, where $\nu_1,\nu_2\in\mathfrak M^+$, is Borel (hence universally measurable), the potential of a positive measure being l.s.c.\ on $X$.} and consequently
\begin{equation}\label{upint}
\int U^{\nu_1}\,d\mu=\int U^{\nu_2}\,d\mu
\end{equation}
(see \cite[Section~IV.2, Proposition~6]{B2}).

\begin{lemma}\label{str}
For any $E\subset X$ and any $U_j\in\mathcal U$, $j\in\mathbb N$,
\[c_*\Bigl(\bigcup_{j\in\mathbb N}\,E\cap U_j\Bigr)\leqslant\sum_{j\in\mathbb N}\,c_*(E\cap U_j).\]\end{lemma}

\begin{proof}Since a strictly positive definite kernel is pseudo-positive, cf.\ \cite[p.~150]{F1}, the lemma follows directly from Fuglede \cite{F1} (see Lemma~2.3.5 and the remark after it). For the Newtonian kernel on $\mathbb R^n$, this goes back to Cartan \cite[p.~253]{Ca2}.\end{proof}

\begin{lemma}\label{l2}If a sequence $(\nu_j)\subset\mathcal E$ converges strongly to $\nu_0$, then there exists a subsequence $(\nu_{j_k})$ whose potentials converge to $U^{\nu_0}$ n.e.\ on $X$.\end{lemma}

\begin{proof}This follows from \cite{F1} (the remark on p.~166 and Lemma~2.3.5) by noting that $U^\nu$, where $\nu\in\mathcal E$, is well defined and finite n.e.\ on $X$, see \cite[p.~164]{F1}.\end{proof}

\subsection{Proof of Theorem~\ref{th1}} If $\omega$ is of finite energy, then (i)--(vi) were established (even in a much more general form) in \cite[Theorem~3.1]{Z-arx-22} and \cite[Theorem~3.1]{Z-arx}. The proof is thus reduced to the case of $\omega\in\mathfrak M^+$ satisfying (a)--(c).\footnote{It is worth emphasizing here that the investigation of this case requires substantially different approaches than those utilized in \cite{Z-arx1}--\cite{Z-arx} with $\omega\in\mathcal E^+$, cf.\ the second paragraph in Section~\ref{sec-intr1}.}

We begin by showing that (\ref{1}) holds true for some $\omega^A\in\mathcal E^+(A)$ if so does (\ref{eqpr3}), the opposite being clear from (\ref{upint}). To this end, we first note, by applying Lemma~\ref{l1} and the domination principle, that (\ref{1}) is equivalent to the two relations
\begin{align}\label{e-th1}
 U^{\omega^A}&\geqslant U^\omega\quad\text{n.e.\ on $A$},\\
U^{\omega^A}&=U^\omega\quad\text{$\omega^A$-a.e.\ on $X$}.\label{e-th1'}
\end{align}

Consider now $\omega^A\in\mathcal E^+(A)$ meeting (\ref{eqpr3}), and assume to the contrary that (\ref{e-th1}) fails. Then, there is compact $K\subset A$ such that $U^{\omega^A}<U^\omega$ on $K$ while $c(K)>0$. By making use of \cite{B2} (Section~IV.1, Proposition~10 and Section~IV.2, Theorem~1), we get
$\int U^{\omega^A}\,d\tau<\int U^\omega\,d\tau$ for any nonzero $\tau\in\mathcal E^+(K)$, which however contradicts (\ref{eqpr3}) for $\mu:=\tau$. Thus (\ref{e-th1}) does indeed hold; therefore, by virtue of Lemma~\ref{l1},
\begin{equation}\label{2}
 U^{\omega^A}\geqslant U^\omega\quad\text{$\omega^A$-a.e.\ on $X$}.
\end{equation}

To complete the proof of the above claim, assume that (\ref{e-th1'}) fails. Then, by (\ref{2}), there is compact $Q\subset A$ with $\omega^A(Q)>0$, such that $U^{\omega^A}>U^\omega$ on $Q$. Using (\ref{2}) once again, we conclude in exactly the same way as in the preceding paragraph that $\int U^{\omega^A}\,d\omega^A>\int U^\omega\,d\omega^A$, which is however impossible in view of (\ref{eqpr3}) with $\mu:=\omega^A$.

To verify the statement on the uniqueness in (i), assume (\ref{1}) holds for some $\nu,\nu'\in\mathcal E^+(A)$ in place of $\omega^A$. Then, because of Lemma~\ref{str}, $U^\nu=U^{\nu'}$ n.e.\ on $A$, hence $(\nu+\nu')$-a.e., and consequently everywhere on $X$, by the domination principle applied to each of $\nu,\nu'$. Therefore, by the energy principle, $\nu=\nu'$ (cf.\ \cite[Lemma~3.2.1]{F1}).

We next aim to show that the solution $\omega^A$ to the minimum $f$-weighted energy problem (\ref{hatw'}) exists if and only if there is the (unique) $\mu_0\in\mathcal E^+_f(A)$ satisfying (i), and then necessarily $\mu_0=\omega^A$.\footnote{In view of the above, this would necessarily imply that the solution to problem (\ref{hatw'}) is unique. The same can alternatively be verified by exploiting the convexity of the class $\mathcal E^+_f(A)$, the parallelogram identity in the pre-Hilbert space $\mathcal E$, and the energy principle (cf.\ \cite[Proof of Lemma~6]{Z5a}).}

Assume first that such $\omega^A$ exists. To check (\ref{e-th1}), suppose to the contrary that there is a compact set $K\subset A$ with $c(K)>0$, such that $U^{\omega^A}<U^\omega$ on $K$. A direct verification then shows that for any nonzero $\tau\in\mathcal E^+(K)$ and any $t\in(0,\infty)$,\footnote{In (\ref{eqpr4}) as well as in (\ref{eqpr4'}) and (\ref{eqpr4''}), we utilize \cite[Section~IV.4, Corollary~2 to Theorem~1]{B2}.}
\begin{equation}\label{eqpr4}
I_f(\omega^A+t\tau)-I_f(\omega^A)=2t\int\bigl(U^{\omega^A}-U^\omega\bigr)\,d\tau+t^2\|\tau\|^2.
\end{equation}
As $0<\|\tau\|<\infty$, the value on the right in (\ref{eqpr4}) (hence, also that on the left) is ${}<0$ when $t>0$ is small enough, which is however impossible, for $\omega^A+t\tau\in\mathcal E^+_f(A)$.

Having thus proved (\ref{e-th1}), we see, again by use of Lemma~\ref{l1}, that for this $\omega^A$,
\begin{equation}\label{eqpr5}
U^{\omega^A}\geqslant U^\omega\quad\text{$\omega^A$-a.e.\ on $X$.}
\end{equation}

Assuming now that (\ref{e-th1'}) fails, we deduce from (\ref{eqpr5}) that there exists a compact set $Q\subset A$ with $\omega^A(Q)>0$, such that $U^{\omega^A}>U^\omega$ on $Q$. Denoting $\upsilon:=\omega^A|_Q$, we have $\upsilon\in\mathcal E^+\setminus\{0\}$ as well as $\omega^A-t\upsilon\in\mathcal E^+_f(A)$ for all $t\in(0,1]$, and therefore
\begin{equation}\label{eqpr4'}0\leqslant I_f(\omega^A-t\upsilon)-I_f(\omega^A)=-2t\int\bigl(U^{\omega^A}-U^\omega\bigr)\,d\upsilon+t^2\|\upsilon\|^2,\end{equation}
which is however again impossible when $t>0$ is sufficiently small.

For the remaining "if" part of the above claim, assume that (\ref{1}) is fulfilled by some (unique) $\mu_0\in\mathcal E^+_f(A)$. To show that the same $\mu_0$ then necessarily serves as the solution $\omega^A$ to problem (\ref{hatw'}), we only need to verify that
\begin{equation}\label{eqpr6}
I_f(\mu)-I_f(\mu_0)\geqslant0\quad\text{for any $\mu\in\mathcal E^+_f(A)$}.
\end{equation}
Clearly,
\begin{equation}
  I_f(\mu)-I_f(\mu_0)=\|\mu-\mu_0\|^2+2\int U^{\mu_0}\,d(\mu-\mu_0)-2\int U^\omega\,d(\mu-\mu_0).\label{eqpr4''}
\end{equation}
By (\ref{1}) with $\mu_0$ in place of $\omega^A$, $U^{\mu_0}=U^\omega$ holds $(\mu+\mu_0)$-a.e.\ (Lemma~\ref{l1}); hence,
\[I_f(\mu)-I_f(\mu_0)=\|\mu-\mu_0\|^2\]
(see \cite[Section~IV.2, Proposition~6]{B2}), which yields (\ref{eqpr6}) by the energy principle.

Assertions (i), (iv) and their equivalence will therefore be established once we prove the statement on the solvability in (iv).

Assume first that $A:=K$ is compact. As $\widehat{w}_f(K)\leqslant0$ by (\ref{hatw'}) with $A:=K$, whereas $I_f(0)=0$, we may only consider {\it nonzero} $\mu\in\mathcal E^+(K)$.\footnote{Such $\mu$ do exist since, according to our convention, we only consider $K\in\mathfrak C_A$ with $c(K)>0$.} For each of those $\mu$, there exist $t\in(0,\infty)$ and $\tau\in\mathcal E^+(K)$ such that $\tau(X)=1$ and $\mu=t\tau$, and hence
\begin{equation*}
  I_f(\mu)=t^2\|\tau\|^2-2t\int U^\omega\,d\tau\geqslant t^2\bigl(c(K)^{-1}-2m_Kt^{-1}\bigr),
\end{equation*}
$m_K\in(0,\infty)$ being given by (\ref{ma}) with $A:=K$. Thus $I_f(\mu)>0$ for all $\mu\in\mathcal E^+(K)$ having the property
\[\mu(X)>2m_Kc(K)=:L_K\in(0,\infty),\]
and therefore $\widehat{w}_f(K)$ would be the same if the class $\mathcal E^+_f(K)$ in (\ref{hatw'}) were replaced by
\begin{equation}\label{lk}
 \mathcal E^+_{L_K}(K):=\bigl\{\mu\in\mathcal E^+(K):\ \mu(X)\leqslant L_K\bigr\}.
\end{equation}
That is,
\begin{equation}\label{wk}
\widehat{w}_f(K)=\inf_{\mu\in\mathcal E^+_{L_K}(K)}\,I_f(\mu)=:\widehat{w}_{f,L_K}(K),
\end{equation}
and consequently, by virtue of the strict positive definiteness of the kernel $\kappa$,
\begin{equation}\label{wk'}-\infty<-2m_KL_K\leqslant\widehat{w}_f(K)\leqslant0.\end{equation}

Choose a (minimizing) sequence $(\mu_j)\subset\mathcal E^+_{L_K}(K)$ such that
\[\lim_{j\to\infty}\,I_f(\mu_j)=\widehat{w}_{f,L_K}(K).\]
Being vaguely bounded in view of (\ref{lk}), $(\mu_j)$ is vaguely relatively compact \cite[Section~III.1, Proposition~15]{B2}, and hence it has a subsequence $(\mu_{j_k})$ converging vaguely to some $\mu_0\in\mathfrak M$.\footnote{Due to the second-countability of $X$, the vague topology on $\mathfrak M$ is first-countable (Section~\ref{sec-intr1}).} Actually, $\mu_0\in\mathfrak M^+(K)$, $\mathfrak M^+(K)$ being vaguely closed \cite[Section~III.2, Proposition~6]{B2}. As the mapping $\nu\mapsto I(\nu)$ is vaguely l.s.c.\ on $\mathfrak M^+$ \cite[Lemma~2.2.1(e)]{F1}, whereas the mapping $\nu\mapsto\int f\,d\nu$ is vaguely continuous on $\mathfrak M^+(K)$, the function $f|_K$ being continuous and of compact support, we thus have
\[\widehat{w}_f(K)\leqslant I_f(\mu_0)\leqslant\liminf_{k\to\infty}\,I_f(\mu_{j_k})=\widehat{w}_{f,L_K}(K),\]
which together with (\ref{wk}) proves that $\mu_0$ serves, indeed, as the (unique) solution $\omega^K$ to problem (\ref{hatw'}) with $A:=K$. Note that the same $\omega^K$ solves the problem of minimizing $I_f(\mu)$ over the (smaller) class $\mathcal E^+_{L_K}(K)$.

The existence of $\omega^K$ satisfying (iv), thereby established, implies that for compact $A:=K$, (i) and (iv) do indeed hold, and moreover they are equivalent.

To examine the case of noncompact $A$, we shall now show that the (finite) constant $L_K$, satisfying (\ref{wk}),
might be chosen to be independent of $K\in\mathfrak C_A$.

By (\ref{1}), $U^{\omega^K}=U^\omega$ n.e.\ on $K$, hence $\gamma_{\mathfrak S}$-a.e.\ (Lemma~\ref{l1}), where $\gamma_{\mathfrak S}$ denotes the capacitary measure on $\mathfrak S:=S(\omega^K)$, see \cite[Section~4.1]{F1}.\footnote{For any compact $Q\subset X$, $c(Q)<\infty$ by the energy principle, and hence $\gamma_Q$ does exist.} As $U^{\gamma_{\mathfrak S}}\geqslant1$ n.e.\ on $\mathfrak S$ \cite[Theorem~4.1]{F1}, hence $\omega^K$-a.e., Fubini's theorem \cite[Section~V.8, Theorem~1]{B2} gives
\[\omega^K(X)\leqslant\int U^{\gamma_{\mathfrak S}}\,d\omega^K=\int U^{\omega^K}\,d\gamma_{\mathfrak S}=\int U^\omega\,d\gamma_{\mathfrak S}=\int U^{\gamma_{\mathfrak S}}\,d\omega.\]
Since $U^{\gamma_{\mathfrak S}}\leqslant1$ on $S(\gamma_{\mathfrak S})$ \cite[Theorem~4.1]{F1}, applying Ugaheri's maximum principle yields $U^{\gamma_{\mathfrak S}}\leqslant M=M_{X,\kappa}$ on $X$ (see Section~\ref{sec-intr1}), and therefore
\begin{equation}\label{mass}
 \omega^K(X)\leqslant M\omega(X)=:L<\infty\quad\text{for all $K\in\mathfrak C_A$},
\end{equation}
the measure $\omega$ being bounded according to (a). Using (\ref{mass}), we obtain
\begin{align*}
\widehat{w}_{f,L}(K)\leqslant I_f(\omega^K)&=\widehat{w}_f(K)=\min_{\mu\in\mathcal E^+(K)}\,I_f(\mu)\\
{}&\leqslant\inf_{\mu\in\mathcal E^+_L(K)}\,I_f(\mu)=\widehat{w}_{f,L}(K)\quad\text{for every $K\in\mathfrak C_A$},
\end{align*}
and so $\omega^K$ solves the problem of minimizing $I_f(\mu)$ over the class $\mathcal E^+_L(K)$ as well. Also,
\begin{equation*}-\infty<-2m_AL\leqslant\widehat{w}_f(K)\leqslant0\quad\text{for all $K\in\mathfrak C_A$},\end{equation*}
cf.\ (\ref{wk'}), so that the net $\bigl(\widehat{w}_f(K)\bigr)_{K\in\mathfrak C_A}\subset\mathbb R$ is bounded. Being obviously decreasing, it must be convergent, hence Cauchy in $\mathbb R$.

For any $K,K'\in\mathfrak C_A$ such that $K\leqslant K'$,
\[(\omega^K+\omega^{K'})/2\in\mathcal E^+_L(K'),\]
hence
\[\|\omega^K+\omega^{K'}\|^2-4\int U^\omega\,d(\omega^K+\omega^{K'})\geqslant4\widehat{w}_{f,L}(K'),\]
and applying the parallelogram identity to $\omega^K,\omega^{K'}\in\mathcal E^+$ gives
\begin{equation}\label{Ca}
 \|\omega^K-\omega^{K'}\|^2\leqslant2I_f(\omega^K)-2I_f(\omega^{K'}).
\end{equation}

Being equal to $\bigl(\widehat{w}_f(K)\bigr)_{K\in\mathfrak C_A}$, the net $\bigl(I_f(\omega^K)\bigr)_{K\in\mathfrak C_A}$ is Cauchy in $\mathbb R$; therefore, in view of (\ref{Ca}), the net $(\omega^K)_{K\in\mathfrak C_A}$ is Cauchy in the strong topology on the cone $\mathcal E^+(A)$. The kernel $\kappa$ being perfect by the energy and consistency principles (Section~\ref{sec-intr1}), there exists the (unique) measure $\mu_0\in\mathcal E^+(A)$ such that
\begin{equation}\label{Conv}
 \omega^K\to\mu_0\quad\text{strongly (hence, vaguely) in $\mathcal E^+(A)$ as $K\uparrow A$},
\end{equation}
$\mathcal E^+(A)$ being strongly closed, hence strongly complete (Section~\ref{sec-intr1}).
Moreover, on account of (\ref{mass}) and the vague convergence of $(\omega^K)_{K\in\mathfrak C_A}$ to $\mu_0$, we have
\[\mu_0(X)\leqslant L<\infty,\]
the mapping $\mu\mapsto\mu(X)$ being vaguely l.s.c.\ on $\mathfrak M^+$ \cite[Section~IV.1, Proposition~4]{B2}, and so, actually,
$\mu_0\in\mathcal E_f^+(A)$.

We assert that the same $\mu_0$ serves as the solution $\omega^A$ to problem (\ref{hatw'}).
As shown above, this will follow once we verify (\ref{1}) for $\mu_0$ in place of $\omega^A$, which in turn is reduced to proving
\begin{equation}\label{KK}
 U^{\mu_0}=U^\omega\quad\text{n.e.\ on $K_0$},
\end{equation}
where $K_0\in\mathfrak C_A$ is arbitrarily given. In view of (\ref{Conv}) and the first countability of the strong topology on $\mathcal E$, there is a subsequence $(\omega^{K_j})_{j\in\mathbb N}$ of the net $(\omega^K)_{K\in\mathfrak C_A}$ such that
 \begin{equation*}
 \omega^{K_j}\to\mu_0\quad\text{strongly (hence, vaguely) in $\mathcal E^+(A)$ as $j\to\infty$},
\end{equation*}
and moreover $K_j\supset K_0$ for all $j$.\footnote{If the latter does not hold, we replace $\mathfrak C_A$ by its subset $\mathfrak C_A':=\{K\cup K_0:\ K\in\mathfrak C_A\}$ with the partial order relation inherited from $\mathfrak C_A$, and then apply to $\mathfrak C_A'$ the same arguments as just above.}
Passing if necessary to a subsequence once again, we can also assume, by making use of Lemma~\ref{l2}, that
\begin{equation}\label{JJ}U^{\mu_0}=\lim_{j\to\infty}\,U^{\omega^{K_j}}\quad\text{n.e.\ on $X$}.\end{equation}
Applying now (\ref{1}) to each of those $K_j$, $j\in\mathbb N$, and then letting $j\to\infty$, we get (\ref{KK}), hence (\ref{1}). (Here we have utilized (\ref{JJ}) and Lemma~\ref{str}.)\footnote{Here and in the next two paragraphs, the reference to Lemma~\ref{str} could be replaced by that to \cite[Lemma~2.3.5]{F1}, the sets in question being universally measurable.}

Substituting the equality $\mu_0=\omega^A$ thereby established into (\ref{Conv}) gives
\[
 \omega^K\to\omega^A\quad\text{strongly (hence, vaguely) in $\mathcal E^+(A)$ as $K\uparrow A$},
\]
hence (\ref{c1}) and (\ref{c2}), where $\omega^K$, resp.\ $\omega^A$, is uniquely determined for $K$, resp.\ $A$, by either of (i) or (iv). (As shown above, such $\omega^K$ and $\omega^A$ do exist.)

Moreover, if $K,K'\in\mathfrak C_A$ and $K\leqslant K'$, then, by virtue of (\ref{1}) and Lemma~\ref{str},
\[U^{\omega^K}=U^{\omega^{K'}}=U^{\omega^A}\quad\text{n.e.\ on $K$},\]
hence $\omega^K$-a.e.\ (Lemma~\ref{l1}). Applying the domination principle therefore yields
\[U^{\omega^K}\leqslant U^{\omega^{K'}}\leqslant U^{\omega^A}\quad\text{on $X$},\]
so that $U^{\omega^K}$ increases pointwise on $X$ as $K\uparrow A$, and does not exceed $U^{\omega^A}$. The proof of (\ref{c3}) is thus reduced to
\begin{equation*}
U^{\omega^A}\leqslant\lim_{K\uparrow A}\,U^{\omega^K}\quad\text{on $X$},
\end{equation*}
which however follows directly from (\ref{c2}) in view of the vague semicontinuity of the mapping $(x,\mu)\mapsto U^\mu(x)$ on $X\times\mathfrak M^+$ \cite[Lemma~2.2.1(b)]{F1}.

The proof of (iii) is finalized by verifying that the limit in (\ref{c3}) is unique, the uniqueness of that in (\ref{c1}) or (\ref{c2}) being obvious since the strong topology on $\mathcal E$ as well as the vague topology on $\mathfrak M$ is Hausdorff. Assume (\ref{c3}) holds for some two $\mu_1,\mu_2\in\mathcal E^+(A)$, and fix $K_0\in\mathfrak C_A$. Since the topology on $X$ is second-countable while the potential of a positive measure is l.s.c.\ on $X$, applying \cite[Appendix~VIII, Theorem]{Doob} shows that there is a subsequence $(U^{\omega^{K_j}})$ of the net $(U^{\omega^K})_{K\in\mathfrak C_A}$ such that
\begin{equation}\label{ap}
U^{\omega^{K_j}}\to U^{\mu_i}\quad\text{pointwise on $X$ as $j\to\infty$,}\quad i=1,2.
\end{equation}
Similarly as above, there is no loss of generality in assuming $K_0\subset K_j$ for all $j\in\mathbb N$. As $U^{\omega^{K_j}}=U^\omega$ n.e.\ on $K_j$, hence n.e.\ on $K_0$, we infer from (\ref{ap}), by making use of Lemma~\ref{str}, that
\[U^{\mu_i}=U^\omega\quad\text{n.e.\ on $K_0$}\quad\text{(hence n.e.\ on $A$),}\quad i=1,2.\] But, by virtue of (i), then necessarily $\mu_1=\mu_2$, as was to be proved.

As seen from the above, (iii) does indeed hold, and moreover $\omega^A$, which is uniquely determined by any one of (\ref{c1})--(\ref{c3}), must satisfy each of (i) and (iv).

To check that the same $\omega^A$ also meets (\ref{ii}), fix $\sigma\in\mathcal E^+$. Along with (\ref{1}),
\[U^{\sigma^A}=U^\sigma\quad\text{n.e.\ on $A$}\] (see \cite[Theorem~4.3]{Z-arx1}), hence $\omega^A$-a.e.\ (Lemma~\ref{l1}), and therefore
\[\int U^{\omega^A}\,d\sigma^A=\int U^\omega\,d\sigma^A,\quad\int U^{\sigma^A}\,d\omega^A=\int U^\sigma\,d\omega^A\]
(see \cite[Section~IV.2, Proposition~6]{B2}), which leads to (\ref{ii}) by means of subtraction.

Assertion (ii) as well as its equivalence to any one of (i), (iii), and (iv) will therefore be established once we show that (\ref{ii}) characterizes $\omega^A$ uniquely within $\mathcal E^+(A)$. Assume to the contrary that there exists $\xi\in\mathcal E^+(A)$ such that
\[\int U^\omega\,d\sigma^A=\int U^\xi\,d\sigma\quad\text{for all $\sigma\in\mathcal E^+$}.\]
Subtracting this from (\ref{ii}) gives $\int U^{\xi-\omega^A}\,d\sigma=0$ for all $\sigma\in\mathcal E^+$,
hence
\[\int U^{\xi-\omega^A}\,d(\xi-\omega^A)=0,\]
and consequently $\xi=\omega^A$, by the energy principle.

It is clear from (\ref{1}) that $\omega^A$ satisfying (i) belongs to $\Gamma_{A,\omega}$, the class $\Gamma_{A,\omega}$ being introduced by (\ref{gamma}). To show that the same $\omega^A$ also meets (v), it is thus enough to verify that for any $\nu\in\mathfrak M^+$, the inequality
\begin{equation}\label{1a}
U^\nu\geqslant U^\omega\quad\text{n.e.\ on $A$}
\end{equation}
necessarily yields
\begin{equation}\label{1b}U^\nu\geqslant U^{\omega^A}\quad\text{everywhere on $X$}.\end{equation}
But this is evident, for substituting (\ref{1}) into (\ref{1a}) gives $U^\nu\geqslant U^{\omega^A}$ n.e.\ on $A$, hence $\omega^A$-a.e., which results in (\ref{1b}) by the domination principle.

To complete the proof of (v), assume (\ref{minpot}) holds for $\mu_i\in\Gamma_{A,\omega}\cap\mathcal E^+$, $i=1,2$, in place of $\omega^A$. Then $U^{\mu_1}\geqslant U^{\mu_2}\geqslant U^{\mu_1}$ (thus, $U^{\mu_1}=U^{\mu_2}$) everywhere on $X$, cf.\ (\ref{1b}),  and so $\mu_1=\mu_2$, by the energy principle (cf.\ \cite[Lemma~3.2.1]{F1}).

We next verify that $\omega^A$, uniquely determined by any one of (i)--(v), fulfils (\ref{minen}) as well. To this end, we only need to show that $\|\mu\|\geqslant\|\omega^A\|$ for every $\mu\in\mathcal E^+$ having the property $U^\mu\geqslant U^\omega$ n.e.\ on $A$, or equivalently $U^\mu\geqslant U^{\omega^A}$ n.e.\ on $A$. Since then $U^\mu\geqslant U^{\omega^A}$ holds $\omega^A$-a.e., we get by integration $\langle\mu,\omega^A\rangle\geqslant\|\omega^A\|^2$, hence
\[\|\mu\|\cdot\|\omega^A\|\geqslant\|\omega^A\|^2\]
by the Cauchy--Schwarz inequality,
and dividing by $\|\omega^A\|>0$ yields the claim.

This finalizes the whole proof of Theorem~\ref{th1} by noting that the uniqueness of $\omega^A$, minimizing the energy over  $\Gamma_{A,\omega}\cap\mathcal E^+$, can be derived, as usual, from the convexity of this class by exploiting the pre-Hilbert structure on the space $\mathcal E$.

\subsection{Remark}\label{mass-est} Relation (\ref{mass}), established above for any $\omega\in\mathfrak M^+$ meeting (a)--(c), also holds true for any {\it bounded} $\omega\in\mathcal E^+$, which can be verified in exactly the same manner. Since the mapping $\mu\mapsto\mu(X)$ is vaguely l.s.c.\ on $\mathfrak M^+$, see \cite[Section~IV.1, Proposition~4]{B2}, in view of (\ref{c2}) we therefore obtain
\begin{equation}\label{emm}
 \omega^A(X)\leqslant M_{X,\kappa}\cdot\omega(X)\quad\text{whenever $\omega(X)<\infty$}.
\end{equation}

\subsection{Proof of Theorem~\ref{th1'}} A locally compact space is sec\-ond-count\-able if and only if it is metrizable and $\sigma$-com\-pact (Bourbaki \cite[Section~IX.2, Corollary to Proposition~16]{B3}). Being thus metrizable, the space $X$ is normal and even perfectly normal,\footnote{By Urysohn's theorem \cite[Section~IX.4, Theorem~1]{B3}, a Hausdorff topological space $Y$ is said to be {\it normal} if for any two disjoint closed subsets $F_1,F_2$ of $Y$, there exist disjoint open sets $D_1,D_2$ such that $F_i\subset D_i$ $(i=1,2)$. Further, a normal space $Y$ is said to be {\it perfectly normal} (see Bourbaki \cite[Exercise~7 to Section~IX.4]{B3}) if each closed subset of $Y$ is a countable intersection of open sets.} and hence Fuglede's theories of outer capacities, outer capacitary measures, and capacitability are applicable (see \cite[Sections~4.3--4.5]{F1}). Utilizing \cite[Theorem~4.5]{F1}, we therefore arrive at the following conclusion.

\begin{theorem}\label{l-top}Any Borel subset of a second-countable, locally compact space $X$, endowed with a perfect kernel $\kappa$, is capacitable.
\end{theorem}

If $A$ is Borel, then so is each of the exceptional sets $E$ appearing in Theorem~\ref{th1}. Applying Theorem~\ref{l-top} therefore gives $c^*(E)=c_*(E)=0$, and Theorem~\ref{th1'} follows.

\section{Further properties of inner and outer swept measures}\label{some}

Assuming throughout this section that $X$, $\kappa$, $\omega$, and $A$ are as required in Theorem~\ref{th1}, we first analyze some further properties of the inner balayage $\omega^A$.

\begin{proposition}\label{bal-pot} We have
\begin{align}\label{in-pot}
&U^{\omega^A}\leqslant U^\omega\quad\text{on all of $X$},\\
&I(\omega^A)=I(\omega^A,\omega)\leqslant I(\omega),\label{in-en}\\
&\omega^A(X)\leqslant M_{X,\kappa}\cdot\omega(X).\label{emm'}\end{align}
\end{proposition}

\begin{proof}Since, clearly, $\omega\in\Gamma_{A,\omega}$, (\ref{in-pot}) and $I(\omega^A)\leqslant I(\omega)$ are implied by Theorem~\ref{th1} (see (v) and (vi), respectively). Applying (\ref{eqpr3}) with $\mu:=\omega^A\in\mathcal E^+(A)$, we obtain the equality in (\ref{in-en}). As for (\ref{emm'}), it only needs to be verified in the case $\omega(X)<\infty$, which has already been done in (\ref{emm}).\end{proof}

\begin{proposition}\label{bal-rest}For any $A'\subset A$ with $c_*(A')>0$ and strongly closed $\mathcal E^+(A')$,\footnote{As in Landkof \cite[p.~264]{L}, (\ref{rest}) might be referred to as "{\it the inner balayage with a rest}".}
\begin{equation}\label{rest}
\omega^{A'}=(\omega^A)^{A'},\end{equation}
and therefore
\begin{equation}\label{rest'}U^{\omega^{A'}}\leqslant U^{\omega^A}\quad\text{on $X$}.\end{equation}
\end{proposition}

\begin{proof}
Unless $\omega\in\mathcal E^+$, $U^\omega$ is bounded on $A'$, whereas $U^\omega|_K$ is continuous for every $K\in\mathfrak C_{A'}$, cf.\ (b)--(c); therefore, $\omega^{A'}$ does exist, and moreover it belongs to $\mathcal E^+(A')$ (Theorem~\ref{th1}). The same conclusion holds true for $(\omega^A)^{A'}$, $\omega^A$ being of finite energy (Theorem~\ref{th1}). Noting that, by virtue of (\ref{1}) and Lemma~\ref{str},
\[U^{(\omega^A)^{A'}}=U^{\omega^A}=U^\omega=U^{\omega^{A'}}\quad\text{n.e.\ on $A'$},\]
we arrive at (\ref{rest}) in view of the statement on the uniqueness in Theorem~\ref{th1}(i).

On account of (\ref{rest}), (\ref{in-pot}) applied to $\omega:=\omega^A\in\mathcal E^+$ gives (\ref{rest'}).\footnote{This can also be derived from the inclusion $\mathfrak C_{A'}\subset\mathfrak C_A$ by noting that, because of (\ref{c3}),
\[U^{\omega^A}(x)=\sup_{K\in\mathfrak C_A}\,U^{\omega^K}(x)\quad\text{for all $x\in X$}.\vspace{-4mm}\]}
\end{proof}

\begin{proposition}\label{bal-tot-m} If Frostman's maximum principle holds, then
the inner balayage $\omega^A$ is of minimum total mass in the class $\Gamma_{A,\omega}\cap\mathcal E^+$, i.e.
\begin{equation}\label{eq-t-m}\omega^A(X)=\min_{\mu\in\Gamma_{A,\omega}\cap\mathcal E^+}\,\mu(X).\end{equation}
\end{proposition}

\begin{proof} Since, by Theorem~\ref{th1}(i), $\omega^A\in\Gamma_{A,\omega}\cap\mathcal E^+$, we are reduced to showing that
\begin{equation*}\omega^A(X)\leqslant\mu(X)\quad\text{for any $\mu\in\Gamma_{A,\omega}\cap\mathcal E^+$.}\end{equation*}
But for such $\mu$, according to Theorem~\ref{th1}(v),
\[U^{\omega^A}\leqslant U^\mu\quad\text{everywhere on $X$},\]
and the claimed inequality follows immediately from Deny's principle of positivity of mass in the form stated in \cite[Theorem~2.1]{Z-arx-22}.
\end{proof}

\begin{remark}\label{t-m-nonun} The extremal property (\ref{eq-t-m}) cannot, however, serve as an alternative characterization of the inner balayage $\omega^A$, for it does not determine $\omega^A$ uniquely. Indeed, consider the $\alpha$-Riesz kernel $|x-y|^{\alpha-n}$ of order $\alpha\leqslant2$, $\alpha<n$, on $\mathbb R^n$, $n\geqslant2$, a closed proper subset $A$ of $\mathbb R^n$ that is {\it not $\alpha$-thin at infinity},\footnote{By Kurokawa and Mizuta \cite{KM}, a set $Q\subset\mathbb R^n$ is said to be {\it inner $\alpha$-thin at infinity} if
\begin{equation*}
 \sum_{j\in\mathbb N}\,\frac{c_*(Q_j)}{q^{j(n-\alpha)}}<\infty,
 \end{equation*}
where $q\in(1,\infty)$ and $Q_j:=Q\cap\{x\in\mathbb R^n:\ q^j\leqslant|x|<q^{j+1}\}$. See also \cite[Section~2]{Z-bal2}.} and a nonzero bounded measure $\omega\in\mathcal E^+$ with $S(\omega)\subset A^c$. Then
\begin{equation}\label{eq-t-m1}\omega^A\ne\omega\quad\text{and}\quad\omega^A(\mathbb R^n)=\omega(\mathbb R^n),\end{equation}
the former being obvious e.g.\ from $\omega^A\in\mathcal E^+(A)$, and the latter from \cite[Corollary~5.3]{Z-bal2}.
Noting that $\omega,\omega^A\in\Gamma_{A,\omega}\cap\mathcal E^+$, we conclude by combining (\ref{eq-t-m}) with (\ref{eq-t-m1}) that there are actually infinitely many measures of minimum total mass in $\Gamma_{A,\omega}\cap\mathcal E^+$, for so is every measure of the form $a\omega+b\omega^A$, where $a,b\in[0,1]$ and $a+b=1$.
\end{remark}

\begin{proposition}\label{bal-tm}If $c_*(A)<\infty$ and Frostman's maximum principle holds, then
\begin{equation*}
\omega^A(X)=\int U^\omega\,d\gamma_A,
\end{equation*}
$\gamma_A$ being the inner equilibrium measure on the set $A$, normalized by $\gamma_A(X)=c_*(A)$.
\end{proposition}

\begin{proof} As $\mathcal E^+(A)$ is strongly closed, $\gamma_A\in\mathcal E^+(A)$ \cite[Theorem~7.2]{Z-arx-22}; and moreover, by Frostman's maximum principle, $U^{\gamma_A}=1$ n.e.\ on $A$ \cite[Section~4.1]{F1}. On account of Lemma~\ref{l1} and Theorem~\ref{th1}(i), we thus have $U^{\gamma_A}=1$ $\omega^A$-a.e.\ on $X$, whereas $U^{\omega^A}=U^\omega$ $\gamma_A$-a.e.\ on $X$. Therefore, by Fubini's theorem \cite[Section~V.8, Theorem~1]{B2},
\[\omega^A(X)=\int 1\,d\omega^A=\int U^{\gamma_A}\,d\omega^A=\int U^{\omega^A}\,d\gamma_A=\int U^\omega\,d\gamma_A\]
as claimed.
\end{proof}

\begin{proposition}\label{bal-mon}Consider a decreasing sequence $(A_j)_{j\in\mathbb N}$ of quasiclosed sets with the intersection $A$ of nonzero inner capacity, and a measure $\omega\in\mathfrak M^+$, $\omega\ne0$, such that either $\omega\in\mathcal E^+$, or  {\rm(a)--(c)} with $A_1$ in place of $A$ hold true.\footnote{See the beginning of Section~\ref{sec-intr2}.} Then
\begin{align}\label{mon1}
  \omega^{A_j}&\to\omega^A\quad\text{strongly and vaguely in $\mathcal E^+$},\\
  U^{\omega^{A_j}}&\downarrow U^{\omega^A}\quad\text{pointwise q.e.\ on $X$}.\label{mon2}
  \end{align}
  \end{proposition}

\begin{proof}
The existence of $\omega^{A_j}$, $j\in\mathbb N$, and $\omega^A$ is ensured by Theorem~\ref{th1}, a countable intersection of quasiclosed sets being likewise quasiclosed (Fuglede \cite[Lemma~2.3]{F71}).\footnote{Recall that for quasiclosed $A\subset X$, $\mathcal E^+(A)$ is strongly closed \cite[Theorem~2.13]{Z-Oh}.}

Clearly, $\widehat{w}_f(A_j)$ increases as $j\to\infty$, and does not exceed $\widehat{w}_f(A)$, cf.\ (\ref{hatw'}). On account of (\ref{hatw}), this gives
\begin{equation*}
-\infty<\lim_{j\to\infty}\,I_f(\omega^{A_j})\leqslant0,
\end{equation*}
and hence the sequence $\bigl(I_f(\omega^{A_j})\bigr)$ is Cauchy in $\mathbb R$. As $\omega^{A_{j+1}}\in\mathcal E^+(A_j)$, whereas $\omega^{A_j}$ minimizes $I_f(\mu)$ over $\mu\in\mathcal E^+(A_j)$ (cf.\ Remark~\ref{rem-int}), we see, by utilizing the convexity of the class
$\mathcal E^+(A_j)$ and the parallelogram identity applied to $\omega^{A_j}$ and $\omega^{A_{j+1}}$, that
\[\|\omega^{A_{j+1}}-\omega^{A_j}\|^2\leqslant2I_f(\omega^{A_{j+1}})-2I_f(\omega^{A_j}),\]
and hence the sequence $(\omega^{A_j})_{j\geqslant k}$ is strong Cauchy in
$\mathcal E^+(A_k)$ for each $k\in\mathbb N$. The cone $\mathcal E^+(A_k)$ being strongly complete (Section~\ref{sec-intr1}), there exists a measure $\mu_0$ which belongs to $\mathcal E^+(A_k)$ for each $k\in\mathbb N$, and such that
\begin{equation}\label{LL}
\omega^{A_j}\to\mu_0\quad\text{strongly (hence, vaguely) in $\mathcal E^+$}.
\end{equation}
We assert that
\begin{equation}\label{EE}
 \mu_0=\omega^A,
\end{equation}
which substituted into (\ref{LL}) would have implied (\ref{mon1}).

As seen from the above, $A^c$ is the union of $(A_k)^c$, $k\in\mathbb N$,
where each $(A_k)^c$ is $\mu_0$-neg\-lig\-ib\-le; thus $A^c$ is likewise $\mu_0$-negligible \cite[Section~IV.4, Proposition~8]{B2}, hence \[\mu_0\in\mathcal E^+(A).\] By virtue of Theorem~\ref{th1}(i), (\ref{EE}) will therefore follow once we verify that
\begin{equation}\label{hryu}
U^{\mu_0}=U^\omega\quad\text{n.e.\ on $A$.}
\end{equation}
But, again by Theorem~\ref{th1}(i), $U^{\omega^{A_j}}=U^\omega$ n.e.\ on $A_j$, hence also n.e.\ on the (smaller) set $A$, which in view of the strong convergence of $(\omega^{A_j})$ to $\mu_0$ results in (\ref{hryu}), thereby establishing (\ref{mon1}). (Here we have used Lemmas~\ref{str} and \ref{l2}.)

As clear from (\ref{rest'}), the sequence $(U^{\omega^{A_j}})$ decreases pointwise on $X$, and moreover
\begin{equation}\label{hr}
 \lim_{j\to\infty}\,U^{\omega^{A_j}}\geqslant U^{\omega^A}\quad\text{on $X$}.
\end{equation}
Since $\omega^{A_j}\to\omega^A$ strongly, applying \cite{F1} (see Lemma~4.3.3 and the paragraph after it) shows that equality actually prevails in (\ref{hr}) q.e.\ on $X$, which is (\ref{mon2}).\end{proof}

The following assertion on the outer balayage is obtained by combining the above propositions with Theorem~\ref{th1'}.

\begin{proposition}\label{pr-ou}
For any one of Propositions~{\rm\ref{bal-pot}--\ref{bal-tot-m}, \ref{bal-tm}, and \ref{bal-mon}}, if all the sets appearing in it are Borel, then such a proposition remains valid with the inner swept measures replaced throughout by the outer swept measures.
\end{proposition}

\section{Applications of balayage to minimum energy problems}\label{sec-appl}

$\bullet$ In what follows, a locally compact space $X$, a kernel $\kappa$ on $X$, a measure $\omega\in\mathfrak M^+$, $\omega\ne0$, and a set $A\subset X$ with $c_*(A)>0$ are as required in Theorem~\ref{th1}. That is: $X$ is second-countable; $\kappa$ is positive, perfect, and satisfies the domination and Ugaheri maximum principles; the cone $\mathcal E^+(A)$ is strongly closed (hence strongly complete); and either $\omega$ has finite energy, i.e.\ $\omega\in\mathcal E^+$, or it is a bounded measure whose potential $U^\omega$ is bounded on $A$, whereas $U^\omega|_K$ is continuous for every compact subset $K\subset A$.

The inner balayage $\omega^A$ of $\omega$ onto $A$, uniquely introduced by means of Theorem~\ref{th1}, is shown below to be an efficient tool in {\it the inner Gauss variational problem},\footnote{For the bibliography on this problem, see e.g.\ \cite{BHS,Dr0}, \cite{O}--\cite{Z5a}, \cite{Z-Rarx,Z-Oh} and references therein.} the problem on the existence of $\lambda_{A,f}\in\breve{\mathcal E}^+(A)$ with
\begin{equation}\label{G}
 I_f(\lambda_{A,f})=\inf_{\mu\in\breve{\mathcal E}^+(A)}\,I_f(\mu)=:w_f(A).
\end{equation}
Here $I_f(\mu)$ is the energy of $\mu$ evaluated in the presence of the external field $f=-U^\omega$, referred to as the Gauss functional or the $f$-weighted energy, see (\ref{Gf}), while
\[\breve{\mathcal E}^+(A):=\bigl\{\mu\in\mathcal E^+(A):\ \mu(X)=1\bigr\}.\]

Since $f$ is obviously $\mu$-integrable for each $\mu\in\breve{\mathcal E}^+(A)$, we have
\begin{equation}\label{ww}
-\infty<\widehat{w}_f(A)\leqslant w_f(A)<\infty,
\end{equation}
the first inequality being clear from (\ref{hatw}). Thus $w_f(A)$ is {\it finite}, which allows us to prove, by use of the convexity of the class $\breve{\mathcal E}^+(A)$ and the pre-Hilbert structure on the space $\mathcal E$, that the solution $\lambda_{A,f}$ to problem (\ref{G}) is {\it unique} (see \cite[Lemma~6]{Z5a}).

As for the existence of $\lambda_{A,f}$, it does indeed exist if $A:=K$ is compact while $f|_K$ is l.s.c., for then  $\breve{\mathcal E}^+(K)$ is vaguely compact \cite[Section~III.1.9, Corollary~3]{B2} while $I_f(\mu)$ is vaguely l.s.c.\ on $\breve{\mathcal E}^+(K)$, cf.\ \cite[Theorem~2.6]{O}. But if any of these two requirements is not fulfilled, then the above arguments, based on the vague topology only, fail down, and the problem becomes "rather difficult" (Ohtsuka \cite[p.~219]{O}).

Our analysis of the inner Gauss variational problem for $A$ and $f=-U^\omega$, indicated at the top of this section, is based on the concept of inner balayage $\omega^A$, introduced in the present study, and is mainly performed with the aid of the approach originated in our recent work \cite{Z-Oh}. However, \cite{Z-Oh} was only concerned with external fields created by measures of {\it finite} energy, and exactly this circumstance made it possible to exploit efficiently the two topologies~--- strong and vague.

Nonetheless, in view of the fact that the inner balayage $\omega^A$ is of the class $\mathcal E^+(A)$ (Theorem~\ref{th1}), the approach suggested in \cite{Z-Oh} can largely be adapted to the external field $f$ in question, which enables us to improve substantially the results from \cite{Z-Oh}, by strengthening their formulations and/or by extending the area of their validity. See Sections~\ref{ssec-solv}--\ref{ssec-sup} for the results thereby obtained, and Section~\ref{sec-prr} for their proofs.

\subsection{On the solvability of problem (\ref{G})}\label{ssec-solv} This subsection is to provide necessary and/or sufficient conditions for the existence of the minimizer $\lambda_{A,f}$.

\begin{theorem}\label{thap}The following {\rm(i$_1$)--(iii$_1$)} on the solvability of problem {\rm(\ref{G})} hold true.
\begin{itemize}
\item[{\rm(i$_1$)}] For $\lambda_{A,f}$ to exist, it is sufficient that
\begin{equation*}c_*(A)<\infty\quad\text{or}\quad\omega^A(X)=1.\end{equation*}
In the latter case,
\begin{equation}\label{eqq}
\lambda_{A,f}=\omega^A.
\end{equation}
\item[{\rm(ii$_1$)}] $\lambda_{A,f}$ fails to exist if the two assumptions are fulfilled:
\begin{equation}\label{nec}
 c_*(A)=\infty\quad\text{and}\quad\omega^A(X)<1.
\end{equation}
\item[{\rm(iii$_1$)}] If Frostman's maximum principle holds, then $\lambda_{A,f}$ does exist whenever
\begin{equation}\label{geq}
 \omega^A(X)>1.
\end{equation}
\end{itemize}
\end{theorem}

$\P$ Thus, if Frostman's maximum principle holds, then $\lambda_{A,f}$ exists if and only if
\begin{equation}\label{Fiff}c_*(A)<\infty\quad\text{or}\quad\omega^A(X)\geqslant1.\end{equation}

\begin{corollary}\label{cor1}
 If $A$ is quasicompact, then $\lambda_{A,f}$ does exist.
\end{corollary}

\begin{proof} As noted by Fuglede \cite[pp.~127--128]{F71}, $A$ is quasicompact if and only if there is an open set $G\subset X$ with $c(G)$ arbitrarily small such that $A\setminus G$ may be extended to a compact subset of $X$ by adding points of $G$. Thus, $A$ must be of finite outer, hence inner capacity, and applying Theorem~\ref{thap}(i$_1$) we obtain the corollary.\end{proof}

\begin{corollary}\label{cor2}
 If $c_*(A)=\infty$, then $\lambda_{A,f}$ fails to exist whenever
 \begin{equation}\label{le}\omega(X)<1/M,\end{equation}
 $M=M_{X,\kappa}\in[1,\infty)$ being the constant appearing in Ugaheri's maximum principle.
\end{corollary}

\begin{proof} Combining (\ref{le}) with (\ref{emm'}) gives $\omega^A(X)\leqslant M\omega(X)<1$, which according to Theorem~\ref{thap}(ii$_1$) results in the claim.
\end{proof}

\subsection{On the inner $f$-weighted equilibrium constant for a set $A$}\label{ssec-const} As seen from \cite[Theorems~1, 2]{Z5a} (cf.\ also \cite[Theorem~3.5]{Z-Oh}), for $\mu\in\breve{\mathcal E}^+(A)$ to serve as the (unique) solution $\lambda_{A,f}$ to problem (\ref{G}), it is necessary and sufficient that
\begin{equation}\label{ch-1}
U_f^\mu\geqslant\int U^\mu_f\,d\mu\quad\text{n.e.\ on $A$},
\end{equation}
or equivalently
\begin{equation}\label{ch-2}
U_f^\mu=\int U^\mu_f\,d\mu\quad\text{$\mu$-a.e.\ on $X$},
\end{equation}
where $U_f^\mu:=U^\mu+f$ is said to be {\it the $f$-weighted potential} of $\mu$.\footnote{For $f$ in question, the $f$-weighted potential $U_f^\nu$, where $\nu\in\mathcal E^+$, is finite n.e.\ on $A$, which is obvious on account of Lemma~\ref{str} and the fact that $U^\nu$ is finite n.e.\ on $X$ (see \cite[p.~164]{F1}).\label{f1}}
The (finite) constant
\begin{equation}\label{ch-3}c_{A,f}:=\int U_f^{\lambda_{A,f}}\,d\lambda_{A,f}\end{equation} is referred to as {\it the inner $f$-weighted equilibrium constant} for the set $A$.\footnote{Similarly as in \cite[p.~8]{ST}, $c_{A,f}$ might also be termed {\it the inner modified Robin constant}.}

\begin{theorem}\label{thap'} Assume Frostman's maximum principle is fulfilled. Also assume that $\lambda_{A,f}$ exists, or equivalently that {\rm(\ref{Fiff})} holds. Then
 \begin{align}
c_{A,f}&<0\quad\text{if\/ $\omega^A(X)>1$},\label{C2'}\\
c_{A,f}&=0\quad\text{if\/ $\omega^A(X)=1$},\label{C1'}\\
c_{A,f}&>0\quad\text{if\/ $\omega^A(X)<1$}.\label{C3'}
\end{align}
\end{theorem}

\begin{remark} We note by substituting (\ref{eqq}) into (\ref{ch-3}) that (\ref{C1'}) remains valid even if Frostman's maximum principle is replaced by Ugaheri's maximum principle.
\end{remark}

\subsection{Convergence theorems for $\lambda_{A,f}$ and $c_{A,f}$} We next analyze the continuity of $\lambda_{A,f}$ and $c_{A,f}$ under the approximation of $A$ by monotone families of sets. See also Corollaries~\ref{cor-cont} and \ref{cor-cont''} below, specifying the limit relations (\ref{e-contG2}) and (\ref{e-contG2''}).

\begin{theorem}\label{th-contG1}
Under the hypotheses of\/ {\rm(i$_1$)} or {\rm(iii$_1$)} in Theorem~{\rm\ref{thap}},
\begin{equation}\label{e-contG1}
 \lambda_{K,f}\to\lambda_{A,f}\quad\text{strongly and vaguely in $\mathcal E^+(A)$ as $K\uparrow A$},
\end{equation}
and moreover
\begin{equation}\label{e-contG2}
\lim_{K\uparrow A}\,c_{K,f}=c_{A,f}.
\end{equation}
\end{theorem}

\begin{theorem}\label{th-contG2}Consider a decreasing sequence $(A_j)_{j\in\mathbb N}$ of quasiclosed sets with the intersection $A$ of nonzero inner capacity, and a measure $\omega\in\mathfrak M^+$, $\omega\ne0$, such that either $\omega\in\mathcal E^+$, or  {\rm(a)--(c)} with $A_1$ in place of $A$ hold true. Then:
\begin{itemize}
\item[{\rm(i$_2$)}] If $c_*(A_1)<\infty$, then
\begin{equation}\label{e-contG1''}
 \lambda_{A_j,f}\to\lambda_{A,f}\quad\text{strongly and vaguely in $\mathcal E^+$},
\end{equation}
and moreover
\begin{equation}\label{e-contG2''}
\lim_{j\to\infty}\,c_{A_j,f}=c_{A,f}.
\end{equation}
\item[{\rm(ii$_2$)}] If\/ Frostman's maximum principle is met, then {\rm(\ref{e-contG1''})} and {\rm(\ref{e-contG2''})} remain valid whenever $\omega^A(X)>1$.
\end{itemize}
\end{theorem}

\subsection{Alternative characterizations of the solution $\lambda_{A,f}$}\label{ssec-alt} Along with the permanent assumptions listed at the beginning of Section~\ref{sec-appl},

$\bullet$ Throughout Sections~\ref{ssec-alt} and \ref{ssec-sup}, we postulate Frostman's maximum principle. Also assume that $\lambda_{A,f}$ exists, or equivalently that (\ref{Fiff}) holds.

Then one can define
\begin{equation}\label{Lambda}\Lambda_{A,f}:=\bigl\{\mu\in\mathcal E^+: \ U_f^\mu\geqslant c_{A,f}\quad\text{n.e.\ on $A$}\bigr\},\end{equation}
$c_{A,f}$ being the (finite) inner $f$-weighted equilibrium constant. (As follows from footnote~\ref{f1}, this definition makes sense.) The class $\Lambda_{A,f}$ is nonempty, for (see Section~\ref{ssec-const})
\begin{equation}\label{LLL'}
\lambda_{A,f}\in\Lambda_{A,f}.
\end{equation}

\begin{theorem}\label{thap''}In addition to the above-mentioned hypotheses, assume that\footnote{As follows from (\ref{emm'}), (\ref{req''}) does hold whenever $\omega(X)\leqslant1/M_{X,\kappa}$, $M_{X,\kappa}\in[1,\infty)$ being the constant appearing in Ugaheri's maximum principle. We also observe by combining (\ref{Fiff}) and (\ref{req''}) that under the assumptions of Theorem~\ref{thap''}, either of $c_*(A)<\infty$ or $\omega^A(X)=1$ must occur.}
\begin{equation}\label{req''}
\omega^A(X)\leqslant1.
\end{equation}
Then the solution $\lambda_{A,f}$ to problem {\rm(\ref{G})} can be represented in the form
\begin{equation}\label{RRR}\lambda_{A,f}=\left\{
\begin{array}{cl}\omega^A+c_{A,f}\gamma_A&\text{if \ $c_*(A)<\infty$},\\
\omega^A&\text{otherwise},\\ \end{array} \right.
\end{equation}
where $\gamma_A$ is the inner equilibrium measure on the set $A$, normalized by $\gamma_A(X)=c_*(A)$. The same $\lambda_{A,f}$ can  be characterized by any one of the following three assertions.
\begin{itemize}
\item[{\rm(i$_3$)}] $\lambda_{A,f}$ is the unique measure of minimum potential in the class $\Lambda_{A,f}$, i.e.
\begin{equation*}U^{\lambda_{A,f}}=\min_{\mu\in\Lambda_{A,f}}\,U^\mu\quad\text{on $X$}.\end{equation*}
\item[{\rm(ii$_3$)}] $\lambda_{A,f}$ is the unique measure of minimum energy norm in the class $\Lambda_{A,f}$, i.e.
\[\|\lambda_{A,f}\|=\min_{\mu\in\Lambda_{A,f}}\,\|\mu\|.\]
\item[{\rm(iii$_3$)}] $\lambda_{A,f}$ is uniquely determined within $\mathcal E^+(A)$ by the equality
\begin{equation*}
U^{\lambda_{A,f}}_f=c_{A,f}\quad\text{n.e.\ on $A$}.
\end{equation*}
Furthermore,
\begin{align}\label{C1}
  c_{A,f}&\geqslant0\quad\text{if\/ $c_*(A)<\infty$},\\
  c_{A,f}&=0\quad\text{otherwise}.\label{C2}
\end{align}
\end{itemize}
\end{theorem}

\begin{corollary}\label{cor-tmass}
 Under the assumptions of Theorem~{\rm\ref{thap''}}, $\lambda_{A,f}$ is of minimum total mass in the class $\Lambda_{A,f}$, i.e.
 \begin{equation*}
 \lambda_{A,f}(X)=\min_{\mu\in\Lambda_{A,f}}\,\mu(X).
 \end{equation*}
\end{corollary}

\begin{remark}
  Unlike to any one of (i$_3$)--(iii$_3$), Corollary~\ref{cor-tmass} cannot serve as an alternative characterization of the solution $\lambda_{A,f}$, for it does not determine $\lambda_{A,f}$ uniquely. On account of Remark~\ref{t-m-nonun}, this is clear from formulas (\ref{RRR}) and (\ref{C1}).
\end{remark}

\begin{corollary}\label{cor-cont}If {\rm(\ref{req''})} is fulfilled, then {\rm(\ref{e-contG2})} can be specified as follows:
\begin{equation}\label{e-contG2'}
c_{K,f}\downarrow c_{A,f}\quad\text{in $\mathbb R$ as $K\uparrow A$}.
\end{equation}
\end{corollary}

\begin{corollary}\label{cor-cont''}Under the hypotheses of Theorem~{\rm\ref{th-contG2}(i$_2$)}, assume moreover that $\omega^{A_1}(X)\leqslant1$. Then {\rm(\ref{e-contG2''})} can be specified as follows:
\begin{equation}\label{e-contG2'''}
c_{A_j,f}\uparrow c_{A,f}\quad\text{in $\mathbb R$ as $j\to\infty$}.
\end{equation}
\end{corollary}

\subsection{On the support of $\lambda_{A,f}$}\label{ssec-sup} Ohtsuka \cite[p.~284]{O} asked when for compact $K\subset X$,
\begin{equation}\label{eq-O}S(\lambda_{K,f})=K.\end{equation}
The following example shows that for the $\alpha$-Riesz kernel, an answer depends on $\alpha$.

\begin{example}
  Assume $X:=\mathbb R^n$, $n\geqslant2$, $\kappa(x,y):=|x-y|^{\alpha-n}$, where $\alpha<n$, $\alpha\in(0,2]$, and consider compact $K\subset\mathbb R^n$ such that for every $x\in K$ and every open neighborhood $U_x$ of $x$, we have $c(K\cap U_x)>0$.\footnote{Such $K$ is said to coincide with its {\it reduced kernel}, cf.\ \cite[p.~164]{L}.} Fix $y\in K^c$, and define $\omega:=\varepsilon_y$. If moreover $\alpha<2$, then $\omega^K(\mathbb R^n)<1$ (see e.g.\ \cite[Corollary~5.3]{Z-bal2}), and hence, by virtue of Theorems~\ref{thap'} and \ref{thap''}, $\lambda_{K,f}$ can be represented in the form
  \[\lambda_{K,f}=\omega^K+c_{K,f}\gamma_K,\quad\text{where \ }c_{K,f}>0.\]
 Noting that $S(\omega^K)=S(\gamma_K)=K$,\footnote{More about the supports of Riesz equilibrium and swept measures, also for noncompact sets, see \cite[Theorems~7.2, 8.5]{Z-bal} and \cite[Corollary~5.4]{Z-bal2}.\label{fsup}} we thus get (\ref{eq-O}) as required.  However, this is no longer so if $\alpha=2$. In particular, if $K^c$ is connected while ${\rm Int}_{\mathbb R^n}K\ne\varnothing$, then $S(\lambda_{K,f})=\partial_{\mathbb R^n}K$ (cf.\ footnote~\ref{fsup}), and hence $S(\lambda_{K,f})$ is a {\it proper} subset of $K$.
  \end{example}

In applications of problem (\ref{G}) to constructive function theory, it is often useful to know when for noncompact $A$, $S(\lambda_{A,f})$ is compact (see e.g.\ \cite{Dr0,Oo,ST}).

\begin{theorem}\label{thapsupport}Assume that $A$ is unbounded, $\omega^A(X)>1$, and
\begin{equation}\label{LLL}
\lim_{x\to\infty_X, \ x\in A}\,U^\omega(x)=0,
\end{equation}
$\infty_X$ being the Alexandroff point of $X$. Then $\lambda_{A,f}$ is of compact support.\footnote{$\lambda_{A,f}$ does exist because of the hypothesis $\omega^A(X)>1$, see Theorem~\ref{thap}(iii$_1$).}\end{theorem}

As seen from Example~\ref{Ex1}, Theorem~\ref{thapsupport} is sharp in the sense that, in general, it would fail to hold if $\omega^A(X)>1$ were replaced by $\omega^A(X)=1$.

\begin{example}\label{Ex1}
Consider the $\alpha$-Riesz kernel $|x-y|^{\alpha-n}$, $\alpha\in(0,2)$, on $\mathbb R^n$, $n\geqslant2$, a closed set $A\subset\mathbb R^n$ that is not $\alpha$-thin at infinity and coincides with its reduced kernel, and a probability measure $\omega$ compactly supported in $A^c$. Then (\ref{LLL}) obviously holds, whereas $\omega^A(\mathbb R^n)=\omega(\mathbb R^n)=1$ \cite[Corollary~5.3]{Z-bal2}. By virtue of Theorem~\ref{thap}(i$_1$), $\lambda_{A,f}$ does exist, and moreover $\lambda_{A,f}=\omega^A$; hence, $S(\lambda_{A,f})=A$ (see \cite[Corollary~5.4]{Z-bal2}).
\end{example}

\section{Proofs}\label{sec-prr}

Assume that $X$, $\kappa$, $A$, and $\omega$ are as required in Theorem~\ref{th1} (cf.\ the beginning of Section~\ref{sec-appl}).
It is crucial to note that then, according to this theorem,
\begin{equation}\label{mm'}\omega^A\in\mathcal E^+(A).\end{equation}

Along with $f=-U^\omega$, consider the external field
\[\tilde{f}:=-U^{\omega^A}.\]
Then $f=\tilde{f}$ n.e.\ on $A$ (see Theorem~\ref{th1}(i)); therefore, by virtue of (\ref{upint}),
\begin{equation}\label{if}
I_f(\mu)=I_{\tilde{f}}(\mu)\quad\text{for all $\mu\in\mathcal E^+(A)$},
\end{equation}
and consequently
\begin{equation}\label{wf}
w_f(A)=w_{\tilde{f}}(A):=\inf_{\mu\in\breve{\mathcal E}^+(A)}\,I_{\tilde{f}}(\mu).
\end{equation}
It follows that the (unique) solution $\lambda_{A,f}$ to problem (\ref{G}) exists if and only if there exists $\lambda_{A,\tilde{f}}\in\breve{\mathcal E}^+(A)$ with $I_{\tilde{f}}(\lambda_{A,\tilde{f}})=w_{\tilde{f}}(A)$, and in the affirmative case,
\begin{equation}\label{mm}
 \lambda_{A,f}=\lambda_{A,\tilde{f}}.
\end{equation}
That is, the minimum $\tilde{f}$-weighted problem (\ref{wf}) is {\it dual\/} to the original problem (\ref{G}), the external field $\tilde{f}=-U^{\omega^A}$ being created by $\omega^A\in\mathcal E^+(A)$.

Also note that obviously
\begin{equation}\label{wpp'}
U_f^{\lambda_{A,f}}=U_{\tilde{f}}^{\lambda_{A,\tilde{f}}}\quad\text{n.e.\ on $A$},
\end{equation}
which in view of (\ref{ch-3}) gives
\begin{equation}c_{A,f}=c_{A,\tilde{f}}.\label{wpp}\end{equation}

$\P$ An advantage of problem (\ref{wf}) in comparison with problem (\ref{G}) is that the $\tilde{f}$-weighted energy $I_{\tilde{f}}(\cdot)$ is strongly continuous on the (strongly closed) set $\mathcal E^+(A)$, for
\begin{equation*}
I_{\tilde{f}}(\mu)=\|\mu-\omega^A\|^2-\|\omega^A\|^2\quad\text{for all $\mu\in\mathcal E^+(A)$}.
\end{equation*}
Thus, if a net $(\mu_s)\subset\mathcal E^+(A)$ converges strongly to $\mu_0$, then
\begin{equation*}
\lim_s\,I_{\tilde{f}}(\mu_s)=I_{\tilde{f}}(\mu_0),
\end{equation*}
hence, cf.\ (\ref{if}),
\begin{equation}\label{dualg'}
\lim_s\,I_f(\mu_s)=I_f(\mu_0).
\end{equation}

\subsection{Preliminary results} A net $(\mu_s)$ is said to be {\it minimizing} in problem (\ref{G}) if $(\mu_s)\subset\breve{\mathcal E}^+(A)$ and
\begin{equation}\label{e-min}\lim_s\,I_f(\mu_s)=w_f(A);\end{equation}
let $\mathbb M_f(A)$ stand for the (nonempty) set of all those $(\mu_s)$.

\begin{lemma}\label{au-1}
There is the unique $\xi_{A,f}\in\mathcal E^+(A)$ such that for every $(\mu_s)\in\mathbb M_f(A)$,
\begin{equation}\label{e-au-1}
 \mu_s\to\xi_{A,f}\quad\text{strongly and vaguely in $\mathcal E^+$}.
\end{equation}
This $\xi_{A,f}$ is said to be extremal in problem {\rm(\ref{G})}.
\end{lemma}

\begin{proof}
It is clear from (\ref{if}), (\ref{wf}), and (\ref{e-min}) that
$(\mu_s)\in\mathbb M_{\tilde{f}}(A)$, where $\mathbb M_{\tilde{f}}(A)$ denotes the set of all nets that are minimizing in the (dual) problem (\ref{wf}). Applying \cite[Lemma~4.1]{Z-Oh}, which is possible because of (\ref{mm'}), therefore gives (\ref{e-au-1}) with
$\xi_{A,f}:=\xi_{A,\tilde{f}}$, where $\xi_{A,\tilde{f}}$ is the (unique) measure extremal in problem (\ref{wf}) (see \cite[Lemma~4.1]{Z-Oh}). The remaining claim $\xi_{A,f}\in\mathcal E^+(A)$ is obvious, $\mathcal E^+(A)$ being strongly closed.
\end{proof}

\begin{lemma}\label{au-2}For the extremal measure $\xi_{A,f}$, we have
\begin{equation}\label{e-au-2}
I_f(\xi_{A,f})=w_f(A).
\end{equation}
Hence, problem {\rm(\ref{G})} is solvable if and only if $\xi_{A,f}\in\breve{\mathcal E}^+(A)$, or equivalently
\begin{equation}\label{e-au-3}
\xi_{A,f}(X)=1,
\end{equation}
and in the affirmative case,
\begin{equation}\label{e-au-4}
\xi_{A,f}=\lambda_{A,f}.
\end{equation}
\end{lemma}

\begin{proof} For any $(\mu_s)\in\mathbb M_f(A)$, $\mu_s\to\xi_{A,f}$ strongly (Lemma~\ref{au-1}), which together with  (\ref{dualg'}) and (\ref{e-min}) gives (\ref{e-au-2}). Therefore, the (unique) extremal measure $\xi_{A,f}$ serves as the (unique) minimizer $\lambda_{A,f}$, provided that (\ref{e-au-3}) holds.

To complete the proof, assume $\lambda_{A,f}$ exists. Since the trivial net $(\lambda_{A,f})$ obviously belongs to $\mathbb M_f(A)$, it must converge strongly to both $\lambda_{A,f}$ and $\xi_{A,f}$ (Lemma~\ref{au-1}). This implies (\ref{e-au-4}), the strong topology on $\mathcal E$ being Hausdorff.\end{proof}

\begin{remark}
In view of the vague convergence of $(\mu_s)\in\mathbb M_f(A)$ to $\xi_{A,f}$,
\begin{equation}\label{xim}
 \xi_{A,f}(X)\leqslant\liminf_{s}\,\mu_s(X)=1,
\end{equation}
the mapping $\nu\mapsto\nu(X)$ being vaguely l.s.c.\ on $\mathfrak M^+$ \cite[Section~IV.1, Proposition~4]{B2}. Problem (\ref{G}) is thus solvable if and only if equality prevails in (\ref{xim}).
\end{remark}

\subsection{Proof of Theorem~\ref{thap}} Assume first that $c_*(A)<\infty$. Then \cite[Theorem~1.2]{Z-Oh} applied to the external field $\tilde{f}=-U^{\omega^A}$ with $\omega^A\in\mathcal E^+(A)$, cf.\ (\ref{mm'}), shows that the dual problem (\ref{wf}) is (uniquely) solvable. Hence, so is the original problem (\ref{G}), and moreover their solutions coincide, cf.\ (\ref{mm}).

The proof of (i$_1$) is thus reduced to the case $\omega^A(X)=1$. Together with (\ref{mm'}), this gives $\omega^A\in\breve{\mathcal E}^+(A)$; therefore, on account of (\ref{ww}),
\[w_f(A)\leqslant I_f(\omega^A)=\widehat{w}_f(A)\leqslant w_f(A).\]
Thus problem (\ref{G}) is indeed solvable, and moreover (\ref{eqq}) holds. Also note that then
\begin{equation}\label{c0ag}
 c_{A,f}=\int U_f^{\lambda_{A,f}}\,d\lambda_{A,f}=\int\bigl(U^{\omega^A}-U^\omega\bigr)\,d\omega^A=0,
\end{equation}
for $U^{\omega^A}-U^\omega=0$ n.e.\ on $A$ (Theorem~\ref{th1}(i)), hence $\omega^A$-a.e.\ on $X$ (Lemma~\ref{l1}).

To verify (ii$_1$), assume (\ref{nec}) is fulfilled. As $c_*(A)=\infty$, there exists a sequence $(\tau_j)\subset\breve{\mathcal E}^+(A)$ converging strongly to zero. Define
\[\mu_j:=\omega^A+q\tau_j,\quad\text{where \ }q:=1-\omega^A(X)\in(0,1);\]
then obviously $(\mu_j)\subset\breve{\mathcal E}^+(A)$, and consequently
\begin{equation}\label{doc1}
w_f(A)\leqslant\liminf_{j\to\infty}\,I_f(\mu_j).
\end{equation}
Observing from (\ref{mm'}) and (\ref{if}) that
\[I_f(\mu_j)=I_{\tilde{f}}(\mu_j)=\|\omega^A+q\tau_j\|^2-2\bigl\langle\omega^A,\omega^A+q\tau_j\bigr\rangle=
-\|\omega^A\|^2+q^2\|\tau_j\|^2,\]
and also that
\[I_f(\omega^A)=I_{\tilde{f}}(\omega^A)=\|\omega^A\|^2-2\bigl\langle\omega^A,\omega^A\bigr\rangle=-\|\omega^A\|^2,\]
we get
\[\lim_{j\to\infty}\,I_f(\mu_j)=-\|\omega^A\|^2=I_f(\omega^A)=\widehat{w}_f(A)\leqslant w_f(A),\]
which together with (\ref{doc1}) shows that $(\mu_j)$ belongs to $\mathbb M_f(A)$, and hence converges strongly to the extremal measure $\xi_{A,f}$ (Lemma~\ref{au-1}). Since obviously
\[\mu_j\to\omega^A\quad\text{strongly},\] we have $\xi_{A,f}=\omega^A$, therefore $\xi_{A,f}(X)<1$ (cf.\ (\ref{nec})), whence (ii$_1$) (Lemma~\ref{au-2}).

To prove (iii$_1$), we first remark that, according to \cite[Lemma~4]{Z5a},
\begin{equation}\label{Kcont}
w_f(K)\downarrow w_f(A)\quad\text{as $K\uparrow A$}.
\end{equation}
By virtue of (i$_1$), for every $K\in\mathfrak C_A$ there exists the solution $\lambda_{K,f}$ to problem (\ref{G}) with $A:=K$, $c(K)$ being finite by the energy principle. As seen from (\ref{Kcont}), those $\lambda_{K,f}$ form a minimizing net, i.e.\
\begin{equation}\label{minnett}
(\lambda_{K,f})_{K\in\mathfrak C_A}\in\mathbb M_f(A),
\end{equation}
and hence, in view of Lemma~\ref{au-1},
\begin{equation}\label{mnet1}
\lambda_{K,f}\to\xi_{A,f}\quad\text{strongly and vaguely in $\mathcal E^+(A)$ as $K\uparrow A$}.
\end{equation}
Furthermore, on account of (\ref{ch-1}) and (\ref{ch-3}),
\begin{equation}\label{mnet2}
U_f^{\lambda_{K,f}}\geqslant c_{K,f}:=\int U_f^{\lambda_{K,f}}\,d\lambda_{K,f}\quad\text{n.e.\ on $K$}.
\end{equation}

We aim to show that
\begin{equation}\label{mnet4}
\lim_{K\uparrow A}\,\int U_f^{\lambda_{K,f}}\,d\lambda_{K,f}=\int U_f^{\xi_{A,f}}\,d\xi_{A,f}=:C_\xi,
\end{equation}
and moreover
\begin{equation}\label{mnet3}
U_f^{\xi_{A,f}}\geqslant C_\xi\quad\text{n.e.\ on $A$}.
\end{equation}
First, in view of the strong convergence of $(\lambda_{K,f})_{K\in\mathfrak C_A}$ to $\xi_{A,f}$, cf.\ (\ref{mnet1}),
\begin{equation}\label{mnet4'}\lim_{K\uparrow A}\,\|\lambda_{K,f}\|^2=\|\xi_{A,f}\|^2,\end{equation}
and also\footnote{In the first and the third equalities in (\ref{mnet5}), we use the fact that $U^\omega=U^{\omega^A}$ holds true n.e.\ on $A$, hence $(\xi_{A,f}+\lambda_{K,f})$-a.e.\ on $X$, whereas the second equality is derived from (\ref{mnet1}) by applying the Cauchy--Schwarz inequality to $\omega^A$ and $\lambda_{K,f}-\xi_{A,f}$, elements of the pre-Hil\-bert space $\mathcal E$.}
\begin{equation}\label{mnet5}\lim_{K\uparrow A}\,\int U^\omega\,d\lambda_{K,f}=\lim_{K\uparrow A}\,\bigl\langle\omega^A,\lambda_{K,f}\bigr\rangle=\bigl\langle\omega^A,\xi_{A,f}\bigr\rangle=\int U^\omega\,d\xi_{A,f},\end{equation}
which combined with (\ref{mnet4'}) results in (\ref{mnet4}). Taking now limits as $K\uparrow A$ in (\ref{mnet2}), on account of (\ref{mnet4}) and Lemmas~\ref{str} and \ref{l2}, we conclude that (\ref{mnet3}) does indeed hold n.e.\ on $K_0$, where compact $K_0\subset A$ is arbitrarily chosen, and hence n.e.\ on $A$.

Utilizing Frostman's maximum principle adopted in (iii$_1$), we shall now verify that (\ref{geq}) ensures the existence of $\lambda_{A,f}$. Suppose first that $C_\xi\geqslant0$, the (finite) constant $C_\xi$ being introduced in (\ref{mnet4}). In view of (\ref{mnet3}), then
\begin{equation*}U^{\xi_{A,f}}\geqslant U^\omega=U^{\omega^A}\quad\text{n.e.\ on $A$},\end{equation*}
the equality being derived from (\ref{1}) with the aid of Lemma~\ref{str}. Since $\omega^A\in\mathcal E^+(A)$ (Theorem~\ref{th1}), $U^{\xi_{A,f}}\geqslant U^{\omega^A}$ holds true $\omega^A$-a.e.\ (Lemma~\ref{l1}), hence on all of $X$ (the domination principle). Therefore, by virtue of Deny's principle of positivity of mass in the form stated in \cite[Theorem~2.1]{Z-arx-22},
\[\xi_{A,f}(X)\geqslant\omega^A(X)>1,\]
the latter inequality being given by (\ref{geq}). Since this contradicts (\ref{xim}),
\begin{equation}\label{xi0}
 C_\xi<0.
\end{equation}

Noting that (\ref{mnet3}) holds $\xi_{A,f}$-a.e., we obtain by integration
\begin{equation*}
C_\xi=\int U_f^{\xi_{A,f}}\,d\xi_{A,f}\geqslant C_\xi\cdot\xi_{A,f}(X),
\end{equation*}
the equality being valid because of (\ref{mnet4}). Therefore, on account of (\ref{xi0}),
\[\xi_{A,f}(X)\geqslant1,\]
which combined with (\ref{xim}) yields (\ref{e-au-3}), hence (\ref{e-au-4}) (Lemma~\ref{au-2}).
This establishes (iii$_1$), thereby completing the whole proof of the theorem.

\subsection{Proof of Theorem~\ref{th-contG1}} Under the stated assumptions, $\lambda_{A,f}$ and $\lambda_{K,f}$, where $K\in\mathfrak C_A$, do exist (Theorem~\ref{thap}), and moreover
$(\lambda_{K,f})_{K\in\mathfrak C_A}\in\mathbb M_f(A)$ and $\lambda_{A,f}=\xi_{A,f}$ (see (\ref{minnett}) and (\ref{e-au-4}), respectively). Hence, by virtue of (\ref{e-au-1}), (\ref{e-contG1}) does indeed hold, which leads, in turn, to (\ref{e-contG2})
in a way similar to that in (\ref{mnet4'}) and (\ref{mnet5}).

\subsection{Proof of Theorem~\ref{th-contG2}} We first observe that $\lambda_{A_j,f}$, $j\in\mathbb N$, as well as $\lambda_{A,f}$ do exist. Indeed, under the requirements of (i$_2$), all the $A_j$, $j\in\mathbb N$, as well as $A$ are of finite inner capacity, and the claim follows at once from Theorem~\ref{thap}(i$_1$) by noting that a countable intersection of quasiclosed sets is likewise quasiclosed \cite[Lemma~2.3]{F71}. If now the hypotheses of (ii$_2$) hold, then, by virtue of (\ref{emm'}) and (\ref{rest}),
\begin{equation}\label{e-contG1-pr0}\omega^{A_j}(X)\geqslant\bigl(\omega^{A_j}\bigr)^A(X)=\omega^A(X)>1\quad\text{for all $j\in\mathbb N$},\end{equation}
and the existence of $\lambda_{A_j,f}$, $j\in\mathbb N$, and $\lambda_{A,f}$ is ensured by Theorem~\ref{thap}(iii$_1$).

In much the same way as in the proof of Proposition~\ref{bal-mon}, one can see that there exists the unique measure $\mu_0\in\mathcal E^+(A)$ such that
\begin{equation}\label{e-contG1-pr}
 \lambda_{A_j,f}\to\mu_0\quad\text{strongly and vaguely in $\mathcal E^+$ as $j\to\infty$},
\end{equation}
which implies, in a manner similar to that in (\ref{mnet4'}) and (\ref{mnet5}), that
\begin{equation}\label{mnet4-pr}
\lim_{j\to\infty}\,c_{A_j,f}=\lim_{j\to\infty}\,\int U_f^{\lambda_{A_j,f}}\,d\lambda_{A_j,f}=\int U_f^{\mu_0}\,d\mu_0=:C_{\mu_0},
\end{equation}
$c_{A_j,f}$ being introduced by (\ref{ch-3}) with $A:=A_j$. Furthermore, according to (\ref{ch-1}),
\begin{equation}\label{mnet4-pr0}
U_f^{\lambda_{A_j,f}}\geqslant c_{A_j,f}=\int U_f^{\lambda_{A_j,f}}\,d\lambda_{A_j,f}\quad\text{n.e.\ on $A_j$},
\end{equation}
hence n.e.\ on the (smaller) set $A$. Taking now limits in (\ref{mnet4-pr0}), on account of (\ref{e-contG1-pr}), (\ref{mnet4-pr}), and Lemmas~\ref{str} and \ref{l2}, we therefore get
\begin{equation}\label{mnet4-pr1}
U_f^{\mu_0}\geqslant\int U_f^{\mu_0}\,d\mu_0=C_{\mu_0}\quad\text{n.e.\ on $A$}.
\end{equation}
We claim that, actually,
\begin{equation}\label{zero'}\mu_0=\lambda_{A,f}\quad\text{and}\quad C_{\mu_0}=c_{A,f},\end{equation}
which substituted into (\ref{e-contG1-pr}) and (\ref{mnet4-pr}) would have proved both (i$_2$) and (ii$_2$).

Assume first that the hypotheses of (ii$_2$) are fulfilled. Then, by virtue of (\ref{e-contG1-pr0}), $\omega^{A_j}(X)>1$, hence $c_{A_j,f}<0$ (cf.\ the proof of Theorem~\ref{thap}(iii$_1$)), which in view of (\ref{mnet4-pr}) gives
$C_{\mu_0}\leqslant0$. If $C_{\mu_0}=0$, then, by (\ref{mnet4-pr1}), $U^{\mu_0}\geqslant U^{\omega^A}$ n.e.\ on $A$, hence $\omega^A$-a.e.\ (Lemma~\ref{l1}), and therefore everywhere on $X$ (the domination principle). Applying Deny's principle of positivity of mass in the form stated in \cite[Theorem~2.1]{Z-arx-22}, we get
\[\mu_0(X)\geqslant\omega^A(X)>1,\]
which is however impossible, $\mu_0$ being the vague limit of a sequence of probability measures, see (\ref{e-contG1-pr}). Thus, necessarily,
\begin{equation}\label{zero}
C_{\mu_0}<0.
\end{equation}
Integrating (\ref{mnet4-pr1}) with respect to $\mu_0$ we get
\[C_{\mu_0}=\int U_f^{\mu_0}\,d\mu_0\geqslant C_{\mu_0}\cdot\mu_0(X),\]
the equality being valid because of (\ref{mnet4-pr}). Therefore, in view of (\ref{zero}), $\mu_0(X)\geqslant1$, which combined with $\mu_0(X)\leqslant1$ and $\mu_0\in\mathcal E^+(A)$ (see above) shows that, actually,
\[\mu_0\in\breve{\mathcal E}^+(A).\] On account of the characteristic property of $\lambda_{A,f}$ given in Section~\ref{ssec-const}, this together with (\ref{mnet4-pr1}) yields (\ref{zero'}), thereby proving (ii$_2$).

To verify (i$_2$), we need the following assertion, having an independent interest. It can, in fact, be derived from \cite{Z-Oh} (see the proof of Theorem~1.2 therein).

\begin{theorem}\label{breve} If $c_*(A)<\infty$, then the class $\breve{\mathcal E}^+(A)$ along with $\mathcal E^+(A)$ is strongly closed, and hence it is strongly complete.\footnote{If $c_*(A)=\infty$, then the class $\breve{\mathcal E}^+(A)$ is clearly no longer strongly closed, for there is a sequence $(\mu_j)\subset\breve{\mathcal E}^+(A)$ converging strongly and vaguely to $\mu_0=0$. It is worth emphasizing that Theorem~\ref{breve} actually holds true for an arbitrary perfect kernel on an arbitrary locally compact space. See, in particular, \cite[Theorem~3.6]{Z-Rarx}, pertaining to the Riesz kernels $|x-y|^{\alpha-n}$, $\alpha\in(0,n)$, on $\mathbb R^n$, $n\geqslant2$.}
\end{theorem}

Since obviously $\lambda_{A_j,f}\in\breve{\mathcal E}^+(A_k)$ for all $j\geqslant k$, whereas $\breve{\mathcal E}^+(A_k)$ is strongly closed (Theorem~\ref{breve} with $A:=A_k$), in view of (\ref{e-contG1-pr}) we get
$\mu_0(X)=1$. As $\mu_0\in\mathcal E^+(A)$ (see above), we thus actually have $\mu_0\in\breve{\mathcal E}^+(A)$. Combining this with (\ref{mnet4-pr1}) yields (\ref{zero'}), again by exploiting the characteristic property of $\lambda_{A,f}$, given in Section~\ref{ssec-const}. This establishes (i$_2$), thereby completing the proof of the whole theorem.

\subsection{Proof of Theorem~\ref{thap''}} It is crucial to note that
\begin{equation}\label{be}
 (\omega^A)^A=\omega^A.
\end{equation}
Indeed, as seen from Theorem~\ref{th1}(iv) (cf.\ also \cite[Theorem~4.3]{Z-arx1}), the inner balayage $\mu^A$ of any $\mu\in\mathcal E^+$ is actually the orthogonal projection of $\mu$ onto the convex, strongly complete cone $\mathcal E^+(A)$. As $\omega^A\in\mathcal E^+(A)$, see (\ref{mm'}), while the orthogonal projection is unique by the energy principle, cf.\ \cite[Theorem~1.12.3]{E2}, this yields (\ref{be}).

Under the stated hypotheses, Frostman's principle is fulfilled, and $\lambda_{A,f}$ exists. By virtue of Theorem~\ref{thap}, then necessarily either $c_*(A)<\infty$ or $\omega^A(X)\geqslant1$, which combined with (\ref{req''}), (\ref{mm'}), and (\ref{be}) shows that \cite[Theorem~1.6]{Z-Oh} with $\zeta:=\omega^A$ (that is, with $\tilde{f}$ in place of $f$) is applicable. Noting from (\ref{wpp'}) and (\ref{wpp}) that
\[\Lambda_{A,\tilde{f}}=\Lambda_{A,f},\]
cf.\ (\ref{Lambda}), on account of (\ref{mm}), (\ref{wpp}), and (\ref{be}) we arrive at the theorem in question by utilizing \cite[Theorem~1.6]{Z-Oh}.

\subsection{Proof of Corollary~\ref{cor-tmass}} In view of (\ref{LLL'}), it is enough to verify that under the assumptions of Theorem~\ref{thap''},
\[\lambda_{A,f}(X)\leqslant\mu(X)\quad\text{for all $\mu\in\Lambda_{A,f}$},\]
which however follows at once from Theorem~\ref{thap''}(i$_3$) by applying \cite[Theorem~2.1]{Z-arx-22}.

\subsection{Proof of Corollary~\ref{cor-cont}} Under the hypotheses of the corollary, the assumptions of Theorem~\ref{thap''} also do hold for any $K\in\mathfrak C_A$ since, by virtue of Proposition~\ref{bal-rest},
\begin{equation*}\omega^K=(\omega^A)^K,\end{equation*}
and therefore, in consequence of (\ref{emm'}) and (\ref{req''}),
\begin{equation}\label{l???'}\omega^K(X)=(\omega^A)^K(X)\leqslant\omega^A(X)\leqslant1.\end{equation}
Applying (\ref{RRR}) to $K\in\mathfrak C_A$ we get
\begin{equation*}c_{K,f}=\frac{1-\omega^K(X)}{c(K)},
\end{equation*}
and (\ref{e-contG2'}) is thus reduced to showing that $\omega^K(X)\leqslant\omega^{K'}(X)$, where $K,K'\in\mathfrak C_A$ and $K\leqslant K'$, which is however obvious from (\ref{l???'}) applied to $K'$ in place of $A$.

\subsection{Proof of Corollary~\ref{cor-cont''}} Since $\omega^{A_1}(X)\leqslant1$, utilizing (\ref{emm'}) and (\ref{rest}) gives
\begin{equation}\label{l???''}\omega^{A_j}(X)=(\omega^{A_1})^{A_j}(X)\leqslant\omega^{A_1}(X)\leqslant1\quad\text{for all $j\in\mathbb N$},\end{equation}
so that Theorem~\ref{thap''} is applicable to each $A_j$. Thus, by virtue of (\ref{RRR}) with $A:=A_j$,
\begin{equation*}c_{A_j,f}=\frac{1-\omega^{A_j}(X)}{c_*(A_j)}.
\end{equation*}
In view of (\ref{e-contG2''}), (\ref{e-contG2'''}) is reduced to proving that $\bigl(\omega^{A_j}(X)\bigr)$ decreases,
which is however obvious from (\ref{l???''}) applied to $A_{j+1}$
and $A_j$ in place of $A_j$ and $A_1$, respectively.

\subsection{Proof of Theorem~\ref{thap'}} Relation (\ref{C1'}) was given in (\ref{c0ag}), while (\ref{C2'}) follows from (\ref{xi0}) by substituting (\ref{e-au-4}) into (\ref{mnet4}). It remains to consider the case where
\begin{equation}\label{l?}
\omega^A(X)<1.
\end{equation}
Then the assumptions of Theorem~\ref{thap''} are met, and applying (\ref{C1}) and (\ref{C2}) gives $c_{A,f}\geqslant0$. This implies (\ref{C3'}), for if not, $U^{\lambda_{A,f}}=U^\omega$ n.e.\ on $A$ (Theorem~\ref{thap''}(iii$_3$)), and therefore $\lambda_{A,f}=\omega^A$ (Theorem~\ref{th1}(i)), which however contradicts (\ref{l?}).

\subsection{Proof of Theorem~\ref{thapsupport}} Suppose that (\ref{LLL}) and $\omega^A(X)>1$ hold true; then, by virtue of the latter, $\lambda_{A,f}$ does exist (Theorem~\ref{thap}(iii$_1$)), and moreover
\begin{equation}\label{l?'}U^{\lambda_{A,f}}-U^\omega=c_{A,f}\quad\text{$\lambda_{A,f}$-a.e.\ on $X$},\end{equation}
where $c_{A,f}<0$ (see (\ref{ch-2})--(\ref{C2'})). Assume, contrary to the claim, that the support $S(\lambda_{A,f})$ is noncompact. Then, according to (\ref{l?'}), there exists a sequence $(x_j)\subset A$ approaching the Alexandroff point $\infty_X$, and such that
\[0<-c_{A,f}\leqslant-c_{A,f}+U^{\lambda_{A,f}}(x_j)=U^\omega(x_j),\]
which is however impossible because of (\ref{LLL}).

\section{Acknowledgements} This research was supported in part by a grant from the Simons Foundation, the USA (1030291, N.V.Z.).

\end{document}